\documentclass[12pt]{amsart}

\usepackage{amsmath}
\usepackage{graphicx}
\usepackage{epstopdf}
\usepackage{amscd}
\usepackage[margin=1in]{geometry}
\usepackage{amsfonts}
\usepackage{amssymb}
\usepackage{titlesec}

 		\usepackage{pgf}
 		\usepackage{tikz}
		\usetikzlibrary{decorations.pathreplacing}

		\newcommand{\tab}{\indent\hspace{.25in}}
		\newcommand{\Par}{ \par \vspace{0.125in} }

		\newtheorem{theorem}{Theorem}[section]
		\newtheorem{lemma}[theorem]{Lemma}
		\newtheorem{proposition}[theorem]{Proposition}
		\newtheorem{corollary}[theorem]{Corollary}
		\newtheorem{claim}[theorem]{Claim}
		\newtheorem{fact}[theorem]{Fact}

		\newtheoremstyle{example}{8pt}{8pt}%
     {}
     {}
     {\bfseries}
     {}
     {6pt}
     {\thmname{#1}\thmnumber{ #2}\thmnote{ #3}}

   \theoremstyle{example}
   \newtheorem{example}[theorem]{Example}
	\newtheorem{remark}[theorem]{Remark}
	\newtheoremstyle{definition}{8pt}{8pt}%
     {}
     {}
     {\bfseries}
     {}
     {6pt}
     {\thmname{#1}\thmnumber{ #2}\thmnote{ #3}}
	\theoremstyle{definition}
   \newtheorem{definition}[theorem]{Definition}

\titleformat{\section}[block]
{\normalfont\Large\bfseries}{Part\ \thesection :}{4pt}{}
\titleformat{\subsection}
{\normalfont\large\bfseries}{\thesubsection}{4pt}{}
\begin{document}
\pagestyle{plain}
\thispagestyle{empty}
\title[$\mathcal{Z}$-structures on Product Groups]{$\mathcal{Z}$-Structures on Product Groups }
\author{Carrie J. Tirel}\thanks{The contents of this paper have been extracted from the author's dissertation for the degree of Doctor of Philosophy at the University of Wisconsin-Milwaukee, which was written under the supervision of Professor Craig R. Guilbault.}

\begin{abstract}A $\mathcal{Z}$\textit{-structure} on a group $G$, defined by M. Bestvina, is a pair $
(\widehat{X}, Z)$ of spaces such that $\widehat{X}$ is a compact ER, $Z$ is a \begin{math}\mathcal{Z}\end{math}-set in $\widehat{X}$, $G$ acts properly and cocompactly on $X=\widehat{X}\backslash Z$, and the collection of translates of any compact set in $X$ forms a null sequence in $\widehat{X}$.  It is natural to ask whether a given group admits a \begin{math}\mathcal{Z}\end{math}-structure.  In this paper, we will show that if two groups each admit a $\mathcal{Z}$-structure, then so do their free and direct products.
\end{abstract}
\maketitle

\section{Introduction}

\subsection{Preliminaries}

Introduced by M. Bestvina in \cite{bestvina96}, a $\mathcal{Z}$-structure on a group is an extension of the notion of a boundary on a CAT(0) or hyperbolic group to more general groups.  Specifically, a $\mathcal{Z}$-structure mimics not only the concept of compactifying the space on which the group $G$ acts in a particularly nice way, but also the fact that boundaries on CAT(0) and hyperbolic groups satisfy a ``null condition,'' essentially meaning that compact sets get ``small'' as they are pushed toward the boundary by elements of $G$.  Here we will review the definitions of $\mathcal{Z}$- and $\mathcal{EZ}$-structures along with several other preliminary definitions and results which will be used later.
\Par
\textit{Note:}  Bestvina's original definition of a $\mathcal{Z}$-structure implies that $G$ is torsion-free.  A. Dranishnikov generalized the definition in \cite{dranishnikov06} to include groups with torsion.  We will use the more general definition in this paper.
\Par
\textbf{Convention:}  In this paper, we assume that all spaces are locally compact, separable metric spaces.

\begin{definition}  A subspace $A$ of a space $X$ is a \textbf{retract} of $X$ if there exists a map $r:X\rightarrow A$ extending $\text{id}_A:A\rightarrow A$.  A subspace $A$ of a space $X$ is a \textbf{strong deformation retract} of $X$ if there exists a homotopy $H:X\times [0,1]\rightarrow X$ satisfying $H_0\equiv \text{id}_{X}$, $H_1(X)\subseteq A$, and $H_t(A)\equiv\text{id}_A$ for all $t\in[0,1]$.\end{definition}

\begin{definition}  A separable metric space $X$ is an \textbf{absolute retract} (or AR) if, whenever $X$ is embedded as a closed subset of another separable metric space $Y$, its image is a retract of $Y$.  $X$ is an \textbf{absolute neighborhood retract} (or ANR) if, whenever $X$ is embedded as a closed subset of another separable metric space $Y$, some neighborhood of $X$ in $Y$ retracts onto $X$.\end{definition}

\begin{definition}  A space $X$ is a \textbf{Euclidean retract} (or ER) if it can be embedded in some Euclidean space as its retract.  $X$ is a \textbf{Euclidean neighborhood retract} (or ENR) if it can be embedded in some Euclidean space $\mathbb{R}^n$ in such a way that a neighborhood of $X$ in $\mathbb{R}^n$ retracts onto $X$.\end{definition}

We recite here a well-known fact concerning a relationship between ANR's, AR's, and ER's, and two useful properties of AR's:

\begin{fact}  \label{ARER}If $X$ is a finite dimensional space, then the following are equivalent:
\\
\tab (i)  $X$ is an AR.
\\
\tab (ii)  $X$ is a contractible ANR.
\\
\tab (iii)  $X$ is an ER.
\\
\tab (iv)  $X$ is contractible and locally contractible.
\end{fact}

\begin{fact} \label{ARextend}  If $X$ is an AR, and $A$ is a closed subspace of a separable metric space $Y$, then every map $f:A\rightarrow X$ extends to $Y$.\end{fact}

\begin{fact}  \label{retract} Every retract of an AR is an AR.\end{fact}

\begin{fact}(Theorem III.7.10 of \cite{hu})  \label{sdr}Let $X$ be an AR.  Then a closed subspace $A$ of $X$ is an AR if and only if $A$ is a strong deformation retract of $X$.\end{fact}
\textit{Proof of Fact \ref{sdr}.}  The sufficiency follows from Fact \ref{retract}.  To prove the necessity, consider the closed subspace $Q:=\left(X\times\left\{0\right\}\right)\cup\left(A\times I\right)\cup\left(X\times \left\{1\right\}\right)$ of $X\times I$.  Since $A$ is an AR and is closed in $X$, there is a retraction $r:X\supseteq A$.  Define a map $F:Q\rightarrow X$ by taking
\Par
\begin{center}\begin{math}F(x,t):=\left\{\begin{array}{l l}x, & \text{ if } x\in X,t=0
\\
x, & \text{ if } x\in A, t\in I
\\
r(x), & \text{ if } x\in X, t=1\end{array}\right.\end{math}\end{center}

\Par
Since $X$ is an AR and $Q$ is closed in $X\times I$, then $F$ extends to $H:X\times I\rightarrow X$, which is a strong deformation retraction of $X$ to $A$.\hfill$\blacksquare$

\begin{definition}  A closed subset $Z$ of an ANR $\widehat{X}$ is a \textbf{$\mathcal{Z}$-set} in $\widehat{X}$ if there exists a homotopy $H:\widehat{X}\times [0,1]\rightarrow \widehat{X}$ such that $H_0\equiv \text{id}_{\widehat{X}}$ and $H_t(\widehat{X})\cap Z=\emptyset$ for all $t>0$.  In this situation, we will call $H$ a \textbf{$\mathcal{Z}$-set homotopy}.\end{definition}

\begin{definition}  An ANR $\widehat{X}$ is a \textbf{$\mathcal{Z}$-compactification of $X$} if $X\subseteq\widehat{X}$, $\widehat{X}$ is compact, and $Z:=\widehat{X}\backslash X$ is a $\mathcal{Z}$-set in $\widehat{X}$.\end{definition}

\begin{remark}  It is easy to see that if $\widehat{X}$ is a $\mathcal{Z}$-compactification of $X$, then the inclusion map $X\hookrightarrow\widehat{X}$ is a homotopy equivalence.  In fact, the inclusion $U\backslash Z\hookrightarrow U$ is a homotopy equivalence for every open subset $U$ of $\widehat{X}$.\end{remark}

\begin{lemma}  \label{sdrzcompactification}If $\widehat{X}$ is an AR which is a $\mathcal{Z}$-compactification of $X$, then there is a homotopy $\widehat{F}:\widehat{X}\times[0,1]\rightarrow\widehat{X}$ and a base point $x_0\in X$ such that $\widehat{F}_0\equiv\text{id}_{\widehat{X}}$, $\widehat{F}_t(\widehat{X})\cap\partial X=\emptyset$ if $t>0$, $\widehat{F}_1(\widehat{X})=\left\{x_0\right\}$ and $\widehat{F}(x_0,t)=x_0$ for all $t\in[0,1]$.\end{lemma}

\begin{proof}  The proof is analogous to that of Fact \ref{sdr}:
\Par
Since $\widehat{X}$ is a $\mathcal{Z}$-compactification of $X$, there is a homotopy $F:\widehat{X}\times[0,1]\rightarrow\widehat{X}$ satisfying $F_0\equiv \text{id}_{\widehat{X}}$ and $F_t(\widehat{X})\cap\partial X=\emptyset$ whenever $t>0$.  Moreover, since $\widehat{X}$ is contractible, then we may assume that $F$ is a contraction to some base point $x_0\in X$.
\Par
Now $F(\left\{x_0\right\}\times[0,1])\cap\partial X=\emptyset$, so we may choose an open neighborhood $U$ of $F(\left\{x_0\right\}\times[0,1])$ in $X$.
\Par
Define $Q:=\left((X\backslash U)\times[0,1]\right)\cup\left(\left\{x_0\right\}\times[0,1]\right)\cup\left(X\times\left\{1\right\}\right)$, and $H:Q\rightarrow X$ by

\begin{center}\begin{math}H(x,t):=\left\{\begin{array}{l l}F(x,t) & \text{ if } (x,t)\in (X\backslash U)\times[0,1]
\\
x_0 & \text{ if } x=x_0, t\in[0,1]
\\
x_0 & \text{ if } x\in X, t=1\end{array}\right.\end{math}\end{center}

Since $X$ is an AR and $Q$ is closed in $X\times [0,1]$, then $H$ extends to $\widehat{H}:X\times [0,1]\rightarrow X$.  Now the map $\widehat{F}:\widehat{X}\times[0,1]\rightarrow \widehat{X}$ defined by
\Par
\begin{center}\begin{math} \widehat{F}(x,t):=\left\{\begin{array}{l l} F(x,t) & \text{ if }x\in V=\partial X, t\in[0,1]
\\
\widehat{H}(x,t) & \text{ if } x\in X, t\in[0,1]\end{array}\right.\end{math}\end{center}
\Par
has the desired properties.\end{proof}

\begin{lemma}\label{slowedsdr}
Suppose $\widehat{X}$ is an AR which is a $\mathcal{Z}$-compactification of $X$, $\left\{C_i\right\}_{i=1}^{\infty}$ is an exhaustion of $X$ by compact sets satisfying $\overline{C_i}\subseteq\text{int} (C_{i+1})$ for all $i\in\mathbb{N}$, and $\left\{t_i\right\}_{i=1}^{\infty}\subseteq (0,1)$ satisfies $t_i>t_{i+1}$ for all $i\in\mathbb{N}$.  Then there is a $\mathcal{Z}$-set homotopy $F:\widehat{X}\times[0,1]\rightarrow\widehat{X}$ which is a strong deformation retraction of $\widehat{X}$ to a base point $x_0\in X$ and having the additional property that 
\begin{equation*}  F(x,t)=x\text{ whenever } (x,t)\in\bigcup_{i=1}^{\infty}\left(C_i\times[0,t_i]\right)\end{equation*}\end{lemma}

\begin{proof}  Let $\overline{F}:\widehat{X}\times[0,1]\rightarrow\widehat{X}$ be a $\mathcal{Z}$-set homotopy which is a strong deformation retraction of $\widehat{X}$ to $x_0\in X$, as in Lemma \ref{sdrzcompactification}.
\Par
Let $A:=\left(\widehat{X}\times\left\{0,1\right\}\right)\bigcup\Bigl(\partial X\times[0,1]\Bigr)\bigcup\Bigl(\bigcup_{i=1}^{\infty}\left(C_i\times[0,t_i]\right)\Bigr)$ and define $f:A\rightarrow[0,1]$ by
\begin{equation*}f(x,t)=\left\{\begin{array}{l l}0 & \text{if }(x,t)\in\widehat{X}\times\left\{0\right\}\bigcup\Bigl(\bigcup_{i=1}^{\infty}\left(C_i\times[0,t_i]\right)\Bigr)
\\
t & \text{if } (x,t)\in\left(\widehat{X}\times\left\{1\right\}\right)\bigcup\left(\partial X\times[0,1]\right)\end{array}\right.\end{equation*}\

Now, since $A$ is a closed subset of $\widehat{X}\times[0,1]$, then $f$ extends to  $f:\widehat{X}\times[0,1]\rightarrow[0,1]$. 
\Par
Then $F:\widehat{X}\times[0,1]\rightarrow\widehat{X}$ defined by $F(x,t):=\overline{F}(x,f(x,t))$ for all $(x,t)\in\widehat{X}\times[0,1]$ has the required attributes.\end{proof}

\begin{definition}  The action of a group $G$ on a space $X$ is \textbf{proper} if every point $x\in X$ has a neighborhood $U$ satisfying $g(U)\cap U=\emptyset$ for all but finitely many $g\in G$.\end{definition}

\begin{definition}  The action of $G$ on $X$ is \textbf{cocompact} if there is a compactum $K$ in $X$ so that $\bigcup\limits_{g\in G} gK=X$.\end{definition}

\begin{definition}  Suppose $G$ is a group acting properly and cocompactly on $X$, and $\widehat{X}$ is a $\mathcal{Z}$-compactification of $X$.  We say that $\widehat{X}$ satisfies the \textbf{null condition with respect to the action of $G$ on $X$} if the following condition holds:
\\
\tab For any compactum $C$ in $X$ and any open cover $\mathcal{U}$ of $\widehat{X}$, there is a finite subset $\Gamma$ of $G$ so that if $g\in G\backslash \Gamma$, then $gC$ is contained in a single element of $\mathcal{U}$.\end{definition}

\begin{definition}  (Bestvina, \cite{bestvina96}) \label{zstructure}Let $G$ be a group.  A \textbf{$\mathcal{Z}$-structure} on $G$ is a pair $(\widehat{X},Z)$ of spaces such that:
\\
\tab (1)  $\widehat{X}$ is a compact ER.
\\
\tab (2)  $\widehat{X}$ is a $\mathcal{Z}$-compactification of $X:=\widehat{X}\backslash Z$.
\\
\tab (3)  $G$ acts properly and cocompactly on $X:=\widehat{X}\backslash Z$.
\\
\tab (4)  $\widehat{X}$ satisfies the null condition with respect to the action of $G$ on $X$.\end{definition}

\begin{remark} Note that if $G$ admits a $\mathcal{Z}$-structure $(\widehat{X},\partial X)$, then $G$ acts on the contractible, finite-dimensional ANR $X$.
\end{remark}
In \cite{bestvina96}, Bestvina discusses the possibility of requiring that the $G$-action on $X$ extend to an action on $\widehat{X}$.  This variation on the notion of $\mathcal{Z}$-structure was formalized by Farrell and LaFont in the following way:

\begin{definition}  (Farrell-LaFont, \cite{farrelllafont})  The pair $(\widehat{X},Z)$ is an \textbf{$\mathcal{EZ}$-structure} on the group $G$ if $(\widehat{X},Z)$ is a $\mathcal{Z}$-structure on $G$, and the action of $G$ on $X:=\widehat{X}\backslash Z$ extends to an action on $\widehat{X}$.\end{definition}

In general, the conditions from Definition \ref{zstructure} which are most difficult to verify when showing the existence of a $\mathcal{Z}$-structure are (1) and (2).  When proving the theorems in this paper, we found the following to be helpful:

\begin{definition}  For an open cover $\mathcal{U}=\left\{U_{\alpha}\right\}_{\alpha\in A}$ of a space $Z$, we say that a homotopy $H:Z\times [0,1]\rightarrow Z$ is a \textbf{$\mathcal{U}$-homotopy} if for each $z\in Z$, there is an $\alpha\in A$ such that $H\left(\left\{z\right\}\times[0,1]\right)\subseteq U_{\alpha}$.\end{definition}

Similarly, if $(Z,d)$ is a metric space, then we say that $H:Z\times[0,1]\rightarrow Z$ is an \textbf{$\epsilon$-homotopy} if for each $z\in Z$, $\text{diam}_d\left(H(\left\{z\right\}\times[0,1])\right)<\epsilon$.

\begin{definition}  A space $W$ \textbf{$\mathcal{U}$-dominates} [respectively, \textbf{$\epsilon$-dominates}] the space $Z$ if there exist maps $\phi:W\rightarrow Z$ and $\psi:Z\rightarrow W$ such that the composition $\phi\circ\psi:Z\rightarrow Z$ is $\mathcal{U}$-homotopic [respectively, $\epsilon$-homotopic] to $\text{id}_Z$.\end{definition}

\begin{theorem}  \label{hanner}(Hanner \cite{hanner})  Each of the following conditions is a sufficient condition for a space $X$ to be an ANR:
\Par
\tab (a)  For each covering $\mathcal{U}$ of $X$ there is an ANR which $\mathcal{U}$-dominates $X$.
\Par
\tab (b)  For some metric on $X$ there exists for each $\epsilon>0$ an ANR which $\epsilon$-dominates $X$.
\end{theorem}

\begin{definition}  A map $f:X\rightarrow Y$ between metric spaces $(X,d)$ and $(Y,d')$ is an \textbf{$\epsilon$-mapping} if $\text{diam}_d f^{-1}(\left\{y\right\})<\epsilon$ for every $y\in Y$.\end{definition}

\begin{theorem}(See p. 107 of \cite{engelking}) \label{engelking}If $X$ is a compact metric space and for every $\epsilon>0$ there exists an $\epsilon$-mapping $f:X\rightarrow Y$ of $X$ to a compact space $Y$ such that $\dim Y\leq n$, then $\dim X\leq n$.\end{theorem}

\begin{corollary}  If $X$ is a metric space with $\dim X\leq n$, and $\widehat{X}$ is a metric space which is a $\mathcal{Z}$-compactification of $X$, then $\dim \widehat{X}\leq n$.\end{corollary}
\begin{proof}  Let $\epsilon>0$ and let $H:\widehat{X}\times[0,1]$ be a $\mathcal{Z}$-set homotopy.  We may choose $t\in (0,1]$ such that the corestriction $H_t:\widehat{X}\rightarrow H_t(\widehat{X})$ is an $\epsilon$-mapping.  Since $H_t(\widehat{X})\subseteq X$ and $\dim X\leq n$, then $\dim H_t(\widehat{X})\leq n$.  Moreover, $H_t(\widehat{X})$ is compact, so Theorem \ref{engelking} applies.\end{proof}

\subsection{Examples}

While the task (posed by Bestvina) of classifying all groups which admit $\mathcal{Z}$-structures remains open, there are various classes of groups which are known to admit $\mathcal{Z}$-structures:

\begin{example}  \label{CAT(0)} A geodesic space $X$ is \textbf{CAT(0)} if geodesic triangles in $X$ are ``no fatter than'' those in the Euclidean plane.  (See Chapter II.1 in \cite{bh} for more background.)
\Par
If $X$ is a CAT(0) space, the \textbf{visual boundary} of $X$, denoted $\partial X$, is the set of geodesic rays emanating from a chosen base point $x_0$.  This boundary on $X$ is well-defined and independent of the base point.  The \textbf{cone topology} on $\overline{X}:=X\cup\partial X$ has as a basis all open balls $B(x,r)\subseteq X$ and all sets of the form $U(c,r,\epsilon)$, where, given a geodesic ray $c$ based at $x_0$ and $r,\epsilon>0$,
\begin{equation*}  U(c,r,\epsilon):=\left\{x\in X\ |\ d(x,x_0)>r,d(p_r(x),c(r))<\epsilon\right\}\cup\left\{x\in\partial X\ |\ d(p_r(x),c(r))<\epsilon\right\} \end{equation*}
where $p_r$ is the natural projection map to $\overline{B}(x_0,r)$.  These neighborhoods $U(c,r,\epsilon)$ of boundary points contain those points in $\overline{X}$ which are sufficiently far from $x_0$ (i.e. sufficiently close to $\partial X$) and which emanate from $x_0$ at the appropriate ``angle.''
\Par
A metric space $(X,d)$ is \textbf{proper} if every closed metric ball in $X$ is compact.
\Par
A group $G$ is \textbf{CAT(0)} if $G$ acts properly and cocompactly by isometries on a proper CAT(0) space.

\begin{fact} If $G$ is a CAT(0) group acting properly and cocompactly by isometries on the proper CAT(0) space $X$, then $(\overline{X},\partial X)$ is a $\mathcal{Z}$-structure on $G$.\end{fact}

It is easy to see that $\overline{X}$ is a $\mathcal{Z}$-compactification of $X$; $\overline{X}$ can be pulled off of $\partial X$ via a homotopy which runs all the geodesic rays in reverse.
\Par
The following statement follows easily from the CAT(0)-inequality and the fact that $G$ acts isometrically on $X$:
\Par
\tab Given a compact $C\subseteq X$, $r>0$, and $\epsilon>0$, there is a number $R>0$ such that if $gC\cap B(x_0,R)=\emptyset$, then there is some $c\in\partial X$ such that $gC\subseteq U(c,r,\epsilon)$.
\Par
This fact, along with properness of the action of $G$ on $X$, imply that $\overline{X}$ satisfies the null condition with respect to the action of $G$ on $X$.
\Par
In addition, the action of $G$ on $X$ extends naturally to $\partial X$, giving:

\begin{fact}  If $G$ is a CAT(0) group, then $G$ admits an $\mathcal{EZ}$-structure.\end{fact}
\end{example}

\begin{example}  (See \cite{bh} for a more thorough treatment.)  A geodesic metric space $(X,d)$ is \textbf{$\delta$-hyperbolic} (where $\delta\geq 0$) if for any triangle with geodesic sides in $X$, each side of the triangle is contained in the $\delta$-neighborhood of the union of the other two sides.
\Par
A group $G$ is \textbf{hyperbolic} if its Cayley graph is $\delta$-hyperbolic for some $\delta\geq 0$.

\begin{theorem}  (Bestvina-Mess \cite{beme})  If $G$ is a torsion-free hyperbolic group, then $G$ admits a $\mathcal{Z}$-structure.\end{theorem}

The proof in \cite{beme} takes as $X$ an appropriately chosen Rips complex of $G$ and as $\partial X$ the Gromov boundary of $X$.
\end{example}

\begin{example} Systolic groups are groups which act simplicially and cocompactly on simplicial complexes which satisfy a combinatorial version of nonpositive curvature.  In \cite{osajda}, D. Osajda and P. Przytycki show that every systolic group admits an $\mathcal{EZ}$-structure.\end{example}

\begin{example}  Bestvina constructs in \cite{bestvina96} multiple $\mathcal{Z}$-structures on the Baumslag-Solitar group $BS(1,2)=\left\langle x,t\ |\ t^{-1}xt=x^2\right\rangle$, one of which is clearly not an $\mathcal{EZ}$-structure.  It is known that this group is not CAT(0), hyperbolic, or systolic.\end{example}

\subsection{Statement of Main Results}

In this paper, we will prove that, if groups $G$ and $H$ each admit $\mathcal{Z}$-structures, then so do their free and direct products:
\Par
\textbf{Theorem \ref{zstructurethmfreeprod}}  \textit{If both $G$ and $H$ admit $\mathcal{Z}$-structures, then so does $G\ast H$.}
\Par
The proof of Theorem \ref{zstructurethmfreeprod} involves the construction of a tree-like space $W$ on which $G\ast H$ acts properly and cocompactly, and the fabrication of a metric $d$ in such a way that the metric completion $\overline{W}$ of $W$, with $\partial W:=\overline{W}\backslash W$, satisfies the axioms of a $\mathcal{Z}$-structure.  The space $W$ is constructed by gluing copies of $X$ and $Y$ in an equivariant manner, where $(\widehat{X},\partial X)$ and $(\widehat{Y},\partial Y)$ are $\mathcal{Z}$-structures on $G$ and $H$, respectively.
\Par
The ability to extend the action of $G\ast H$ on $W$ to $\overline{W}$ is a consequence of the assumption that the actions of $G$ and $H$ extend to $\widehat{X}$ and $\widehat{Y}$, which allows us to obtain:
\Par
\textbf{Theorem \ref{ezstructurefreeprod}  }  \textit{If $G$ and $H$ each admit $\mathcal{EZ}$-structures, then so does $G\ast H$.}
\Par
The other main results found in this paper pertain to direct products of groups which admit $\mathcal{Z}$-structures.
\Par
\textbf{Theorem \ref{zstructurethmdirprod}  }  \textit{If both $G$ and $H$ admit $\mathcal{Z}$-structures, then so does $G\times H$.}
\Par
The proof of Theorem \ref{zstructurethmdirprod} is motivated by its analog in the CAT(0) setting (See \cite{bh}, Example II.8.11(6)):
\Par
If $X$ and $Y$ are CAT(0) spaces, then so is $X\times Y$ under the Euclidean product metric.  If $\partial X$ and $\partial Y$ denote the visual boundaries of $X$ and $Y$ (see Example \ref{CAT(0)} for definitions), let $\partial X\ast \partial Y$ represent the spherical join of $\partial X$ and $\partial Y$, i.e. $\partial X\ast \partial Y=\partial X\times\partial Y\times[0,\frac{\pi}{2}]/\sim$, where $(c_1,c_2,\theta)\sim(c_1',c_2',\theta)$ if and only if $[\theta=0 \text{ and } c_1=c_1']$ or $[\theta=\frac{\pi}{2} \text{ and } c_2=c_2']$.
\Par
Then $\partial (X\times Y)$ is naturally homeomorphic to $\partial X\ast\partial Y$:
\Par
For each $\theta\in\left[0,\frac{\pi}{2}\right]$, $c_1\in\partial X$, and $c_2\in\partial Y$, denote by $(\cos\theta)c_1 +(\sin\theta)c_2$ the point of $\partial(X\times Y)$ represented by the ray $t\mapsto (c_1(t\cos\theta),c_2(t\sin\theta))$.  Then every point of $\partial(X\times Y)$ can be represented by a ray of this form.  Of course, the rays $(\cos\theta)c_1 +(\sin\theta)c_2$ and $(\cos\theta)c_1' +(\sin\theta)c_2'$ are equal when $\theta=0$ and $c_1=c_1'$ (regardless of whether or not we have $c_2=c_2'$), and when $\theta=\frac{\pi}{2}$ and $c_2=c_2'$.  This is consistent with the equivalence relation defining $\partial X\ast \partial Y$.
\Par
Intuitively, for points in $\overline{X\times Y}$ to be ``close'' to a given boundary point $(c_1,c_2,\theta)\in\partial X\ast\partial Y$, it is not sufficient to have $X$-coordinate near $c_1$ and $Y$-coordinate near $c_2$; they must also have ``angle'' near $\theta$.
\Par
Now, given CAT(0) groups $G$ and $H$ which act properly and cocompactly on $X$ and $Y$, respectively, the pair $(\overline{X\times Y},\partial X\ast \partial Y)$ is a $\mathcal{Z}$-structure on $G\times H$.
\Par
The fact that the null condition with respect to the action of $G\times H$ on $X\times Y$ is satisfied by $\overline{X\times Y}$ is an implication of the CAT(0)-inequality, the general idea being that the span of ``angles'' achieved by a compactum shrinks as it is translated outside of a large metric ball.
\Par
To prove the theorem for general direct products, we define a notion of ``slope'' which cooperates with the given $\mathcal{Z}$-set homotopies and certain carefully chosen metrics on the factors $X$ and $Y$.  In the CAT(0) case, we can take as slope function $\displaystyle (x,y)\mapsto\frac{d_Y(y,y_0)}{d_X(x,x_0)}$ thanks to the properties of the CAT(0) metrics; in the general case, we construct functions $p:X\rightarrow [0,\infty)$ and $q:Y\rightarrow[0,\infty)$ to have similar properties and use these to define slope.
\Par
To compactify $X\times Y$, then, we glue to it the join $\partial X\ast\partial Y$ and topologize with neighborhoods of boundary points analogous to those used in the CAT(0) setting.
\Par
By extending the action of $G\times H$ on $X\times Y$ to the $\mathcal{Z}$-compactification $\widehat{X\times Y}$ described above, we obtain:
\Par
\textbf{Theorem \ref{ezstructuredirprod}  }  \textit{If $G$ and $H$ each admit $\mathcal{EZ}$-structures, then so does $G\times H$.}

\pagestyle{myheadings}
\section{$\mathcal{Z}$- and $\mathcal{EZ}$-Structures on Free Products of Groups}
\pagestyle{headings} \thispagestyle{headings}

\noindent Suppose $(\widehat{X},\partial X)$ and $(\widehat{Y},\partial Y)$ are $\mathcal{Z}$-structures on $G$ and $H$, respectively.
\Par
Let $\rho$ and $\tau$ be metrics on $\widehat{X}$ and $\widehat{Y}$ satisfying $\text{diam}_{\rho}\widehat{X}=\text{diam}_{\tau}\widehat{Y}=1$.
\Par
\textit{Notation:}  (i)  Denote by $1_G$ and $1_H$ the identity elements from $G$ and $H$, respectively, and by $\textbf{1}$ the identity element in $G\ast H$.
\Par
(ii)  Whenever we refer to a word $\textbf{1}\neq w\in G\ast H$, it is always assumed that $w$ is reduced, i.e. that consecutive letters of $w$ come from alternating factors, with no letter being an identity element from either group.  With this in mind, we define, for $w\neq \textbf{1}$:
\\
\tab $\bullet$  $\left|w\right|:=$ the length of $w$
\\
\tab $\bullet$ $w(k):=$ the $k$th letter of $w$, counting from left to right
\\
\tab $\bullet$ $w|_k:=$ the leftmost length $k$ subword of $w$
\Par
(iii)  We will use the convention that $|\textbf{1}|=0$, that $\textbf{1}(|\textbf{1}|)=\textbf{1}(0)=\textbf{1}$, and that $\textbf{1}\in G\cap H$.

\begin{definition}  \textbf{(Definition of $W$)}  Let $X_0$ and $Y_0$ be ``base'' copies of $X$ and $Y$, respectively.  Define
\begin{equation*}W:=\biggl(\Bigl(\bigcup\limits_{\substack{w\in G\ast H \\ w(|w|)\in H}}wX_0\Bigr)\bigcup\Bigl(\bigcup\limits_{\substack{w\in G\ast H \\ w(|w|)\in G}}wY_0\Bigr)\biggl)\bigg /\sim\end{equation*}

To define the equivalence relation $\sim$, first note that if $w(|w|)\in H$, then $wX_0$ contains all the points of the form $wgx_0$ for $g\in G$, including the point $wx_0$.  Similarly, if $w(|w|)\in G$, then $wY_0$ contains all the points of the form $why_0$ for $h\in H$, including the point $wy_0$.
\Par
In other words, if $w(|w|)\in H$, then $wx_0\in wX_0$; otherwise $wx_0\in w|_{|w|-1}X_0$.  Likewise, if $w(|w|)\in G$, then $wy_0\in wY_0$; otherwise $wy_0\in w|_{|w|-1}Y_0$.
\Par
Therefore, we define $\sim$ by
\begin{equation*}wx_0\sim wy_0\text{ for all }w\in G\ast H\end{equation*}

The result of this gluing is that, if $w(|w|)\in H$, then $wX_0$ is glued to $w|_{|w|-1} Y_0$ by identifying the points $w x_0\in w X_0$ and $wy_0\in w|_{|w|-1} Y_0$.  Analogously, if $w(|w|)\in G$, then $wY_0$ is glued to $w|_{|w|-1} X_0$ by identifying the points $w y_0\in w Y_0$ and $wx_0\in w|_{|w|-1} X_0$.  (See Figure \ref{W} below.)\end{definition}

\begin{figure}[!h]
\begin{center}
\includegraphics[height=2.75in]{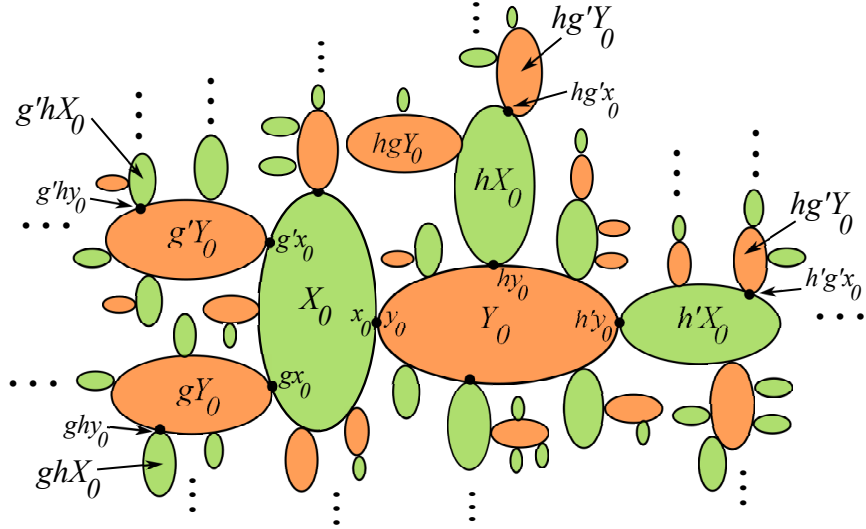}\end{center}
\caption{The gluing schematic for $W$}
\label{W}
\end{figure}

\textit{Warning:}  Although it is included with the intention of providing some intuition about the construction of $W$, Figure \ref{W} has the potential to be a bit misleading, due to its two-dimensionality.  We warn the reader that the points $wx_0$ and $wy_0$ are \textbf{not} boundary points of translates of $X_0$ and $Y_0$, despite their appearance in the graphic.

\begin{remark}\label{remarksW}  (i)  The above construction of $W$ is similar to that found in the proof of Theorem II.11.16 in \cite{bh}; said theorem produces a complete CAT(0) space on which $\Gamma_0\ast_{\Gamma}\Gamma_1$ acts properly [and cocompactly]  by isometries when each of $\Gamma_0,\Gamma_1$, and $\Gamma$ acts properly [and cocompactly] by isometries on a CAT(0) space.  Our construction allows more general spaces but yields essentially the same underlying space under the hypotheses of the cited theorem in the case where $\Gamma$ is trivial.
\Par
(ii)  The action of $G\ast H$ on $W$ is as follows:
\Par
Note that each point of $wX_0$ where $w(|w|)\in H$ has the form $wx$ for some $x\in X_0$.  Thus if $x\in X_0$, we define $w\cdot x:=wx\in wX_0$.
\Par
We define $w\cdot y$ for $y\in Y_0$ and $w(|w|)\in G$ similarly.
\Par
If $x\in X_0$ and $w(|w|)\in G$, then $wX_0=w|_{|w|-1}w(|w|)X_0=w|_{|w|-1}X_0$, and $w\cdot x:=wx=w|_{|w|-1}\cdot w(|w|)x\in w|_{|w|-1}X_0$.  Similarly, if $y\in Y_0$ and $w(|w|)\in H$, then $wY_0=w|_{|w|-1}Y_0$, and $w\cdot y:=wy\in w|_{|w|-1}Y_0$.
\Par
Now for a general point $z\in W$, there is some $x\in X_0$ or $y\in Y_0$ and some $w'\in G\ast H$ such that $z=w'\cdot x$ or $z=w'\cdot y$.  In the first case, we define $w\cdot z:=(ww')\cdot x$; otherwise we set $w\cdot z:=(ww')\cdot y$.
\Par
(iii)  For the rest of this chapter, it is to be understood that the use of the symbol $wX_0$ implies that $w(|w|)\in H$, and the use of the symbol $wY_0$ implies that $w(|w|)\in G$.
\end{remark}

\begin{definition}  \textbf{(Definition of metric $d$ on $W$)}  Define $r:G\cup H\rightarrow \mathbb{N}$ by
\\
\tab $r(g)=n\ \Longleftrightarrow \ gx_0\in B_{\rho}(\partial X,\frac{1}{2^{n-1}})\backslash B_{\rho}(\partial{X},\frac{1}{2^n})$
\\
\tab $r(h)=n\ \Longleftrightarrow\  hy_0\in B_{\tau}(\partial Y,\frac{1}{2^{n-1}})\backslash B_{\tau}(\partial{Y},\frac{1}{2^n})$
\Par
and $r^{\ast}:G\ast H\rightarrow (0,1]$ by
\\
\tab $\textbf{1}\in G\ast H\longmapsto 1$
\\
\tab $g\in G\longmapsto \frac{1}{2^{r(g)}}$
\\
\tab $h\in H\longmapsto\frac{1}{2^{r(h)}}$
\\
\tab $w\in G\ast H\longmapsto\prod\limits_{k=1}^{\left|w\right|}\frac{1}{2^{r(w(k))}}=\prod\limits_{k=1}^{\left|w\right|}r^{\ast}(w(k))$
\Par
We use the function $r^{\ast}$ to define a metric $d$ on $W$:
\Par
The restriction of $d$ to $wX_0$ (respectively $wY_0$) is declared to be a rescaling of $\rho$ (respectively $\tau$) so that $\text{diam}_d wX_0=\text{diam}_d wY_0=r^{\ast}(w)$.
\Par
For points $x,x'\in W$ which do not lie in a single translate of $X_0$ or $Y_0$, we say that a finite sequence $\left\{w_ix_0\right\}_{i=1}^k$ \textbf{\textit{connects}} $x$ and $x'$ if each of the pairs $(x,w_1x_0)$, $(w_kx_0, x')$, and $(w_ix_0,w_{i+1}x_0)$ for $i=1,\ldots,k-1$, lives in a single translate of $X_0$ or $Y_0$.
\Par
To define $d(x,x')$ when $x$ and $x'$ do not lie in a single translate of $X_0$ or $Y_0$, let $\left\{w_ix_0\right\}_{i=1}^k$ be the shortest sequence which connects $x$ and $x'$, and set
\begin{equation*}d(x,x'):=d(x,w_1x_0)+\sum\limits_{i=1}^{k-1}d(w_ix_0,w_{i+1}x_0) + d(w_kx_0,x')\end{equation*}\end{definition}

It is an easy exercise to check that $d$ is indeed a metric on $W$.  The proof that $d$ satisfies the triangle inequality resembles its counterpart for a tree, using in addition the triangle inequality on the components $wX_0$ and $wY_0$.
\Par
Now $(W,d)$ is a metric space, and we denote by $\overline{W}$ the metric completion of $(W,d)$ and set $\partial W:=\overline{W}\backslash W$.
\Par
Let us discuss briefly the convention to be used from this point forward when referring to points of $\partial W$.  We may view $\overline{W}$ as the set of equivalence classes of Cauchy sequences in $W$, where $\sim$ is generated by $\left\{x_i\right\}_{i=1}^{\infty}\sim\left\{x_i'\right\}_{i=1}^{\infty}$ if $d(x_i,x_i')\rightarrow 0$ as $i\rightarrow\infty$.  It is not difficult to see that, if a Cauchy sequence $\vec{x}=\left\{x_i\right\}_{i=1}^{\infty}\subseteq W$ does not converge in $W$, then $\vec{x}\sim\vec{x}'$ where $\vec{x}'=\left\{x'_i\right\}_{i=1}^{\infty}$ falls under one of three possible categories:
\Par
\tab (i)  There exists $w\in G\ast H$ such that $x'_i\in wX_0$ for all $i\in\mathbb{N}$
\\
\tab (ii)  There exists $w\in G\ast H$ such that $x'_i\in wY_0$ for all $i\in\mathbb{N}$
\\
\tab (iii)  There exists a sequence $\left\{w_i\right\}_{i=1}^{\infty}\subseteq G\ast H$ with $\left|w_i\right|=i$, $w_i|_{i-1}=w_{i-1}$, and $x'_i=w_ix_0$ for all $i\in\mathbb{N}$.
\Par
If $\vec{x}'$ falls under category (i), then $x_i,x'_i\rightarrow \overline{x}$ for some $\overline{x}\in\partial wX_0=w\partial X_0$; if $\vec{x}'$ falls under category (ii), then $x_i,x'_i\rightarrow \overline{y}$ for some $\overline{y}\in\partial wY_0=w\partial Y_0$.  Otherwise, $\vec{x}'$ corresponds to a unique element of $\mathcal{A}$, where
\begin{equation*}\mathcal{A}=\biggl\{\overline{\alpha}\ |\ \overline{\alpha}=\left\{\alpha_i\right\}_{i=1}^{\infty}\ ,\ \alpha_i=w_ix_0\ , \ w_i\in G\ast H\ , \ \left|w_i\right|=i\ ,\ w_{i+1}|_i=w_i\ \forall i\in\mathbb{N}\biggr\}\end{equation*}

\noindent Hence, we will refer to points $\alpha\in\partial W$ as having three possible types, each of which corresponds to one of the above-named categories for Cauchy sequences in $W$ which do not converge in $W$:
\Par
\tab (i)  $\alpha\in w\partial X_0$
\\
\tab (ii)  $\alpha\in w\partial Y_0$
\\
\tab (iii)  $\alpha\in\mathcal{A}$
\\
\begin{proposition} \label{Wbarcpt} $\overline{W}$ is compact.\end{proposition}
\begin{proof}  Since $\overline{W}$ is complete, it suffices to show that $\overline{W}$ is totally bounded, i.e. that for any $\epsilon>0$, there is a finite cover of $\overline{W}$ by $\epsilon$-balls.
\Par
Let $\epsilon>0$.  Choose $k\in\mathbb{N}$ such that $\frac{1}{2^{k-1}}<\frac{\epsilon}{4}$.
\Par
Then if a reduced word $w\in G\ast H$ satisfies $\left|w\right|\geq k$, we have
\begin{eqnarray*}\text{diam}_d\Biggl(\biggl(\bigcup_{\substack{w'|_{\left|w\right|}=w \\ w'(\left|w'\right|)\in H}}w'\widehat{X}_0\biggr)\bigcup\biggl(\bigcup_{\substack{w'|_{\left|w\right|}=w \\ w'(\left|w'\right|)\in G}}w'\widehat{Y_0}\biggr)\Biggr)
\leq\sum\limits_{n=k}^{\infty}\frac{1}{2^n}=\frac{1}{2^{k-1}}<\frac{\epsilon}{4}\end{eqnarray*}

Thus if $x\in v\widehat{X}_0$, where $v|_{\left|w\right|}=w$ and $\left|w\right|\geq k$, then
\begin{eqnarray*}d(x,wx_0)&\leq&\text{diam}_d\Biggl(\biggl(\bigcup_{\substack{w'|_{\left|w\right|}=w \\ w'(\left|w'\right|)\in H}}w'\widehat{X}_0\biggr)\bigcup\biggl(\bigcup_{\substack{w'|_{\left|w\right|}=w \\ w'(\left|w'\right|)\in G}}w'\widehat{Y_0}\biggr)\Biggr)+\text{diam}_d w\widehat{X}_0
\\
&\leq&\frac{1}{2^{k-1}}+\frac{1}{2^k}<2\cdot\frac{1}{2^{k-1}}<\frac{\epsilon}{2}\end{eqnarray*}
Similarly, if $y\in v\widehat{Y_0}$, where $v|_{\left|w\right|}=w$ and $\left|w\right|\geq k$, then $d(y,wx_0)<\frac{\epsilon}{2}$.
\Par
Therefore, if $\left|w\right|\geq k$, then
\begin{equation*}\biggl(\Bigl(\bigcup_{\substack{w'|_{\left|w\right|}=w \\ w'(\left|w'\right|)\in H}}w'\widehat{X}_0\Bigr)\bigcup\Bigl(\bigcup_{\substack{w'|_{\left|w\right|}=w \\ w'(\left|w'\right|)\in G}}w'\widehat{Y_0}\Bigr)\biggr)\subseteq B_d(wx_0,\frac{\epsilon}{2})\end{equation*}

\noindent For any $j\in\mathbb{N}$, denote by $\overline{W_j}$ the union of all translates of $\widehat{X}_0$ and $\widehat{Y_0}$ by elements of $G\ast H$ having length no more than $j$, i.e.
\begin{equation*}\overline{W_j}:=\biggl(\Bigl(\bigcup_{\substack{\left|w'\right|\leq j \\ w'(\left|w'\right|)\in H}}w'\widehat{X}_0\Bigr)\bigcup\Bigl(\bigcup_{\substack{\left|w'\right|\leq j \\ w'(\left|w'\right|)\in G}}w'\widehat{Y_0}\Bigr)\biggr)\end{equation*}

Now suppose we have a finite cover $\mathcal{U}$ of $\overline{W_k}$ (where $k$ satisfies $\frac{1}{2^{k-1}}<\frac{\epsilon}{4}$, as earlier) by $\frac{\epsilon}{2}$-balls, and let $\mathcal{U}'$ be the finite cover of $\overline{W_k}$ by $\epsilon$-balls obtained by increasing the radius of each element of $\mathcal{U}$ to $\epsilon$.
\Par
We claim that $\mathcal{U}'$ covers all of $\overline{W}$:
\Par
First, consider a word $w\in G\ast H$ having $\left|w\right|=k$.
Since $wx_0\in w\widehat{X}_0$ or $wx_0\in w\widehat{Y_0}$, and $\left|w\right|=k$, there is some $y\in\overline{W_k}$ such that $wx_0\in B_d(y,\frac{\epsilon}{2})\in\mathcal{U}$.  Then $B_d(y,\epsilon)\in\mathcal{U}'$, and by earlier comments, we have 
\begin{equation*}\biggl(\Bigl(\bigcup_{\substack{w'|_{\left|w\right|}=w \\ w'(\left|w'\right|)\in H}}w'\widehat{X}_0\Bigr)\bigcup\Bigl(\bigcup_{\substack{w'|_{\left|w\right|}=w \\ w'(\left|w'\right|)\in G}}w'\widehat{Y_0}\Bigr)\biggr)\subseteq B_d(wx_0,\frac{\epsilon}{2})\subseteq B_d(y,\epsilon)\end{equation*}

Therefore
\begin{equation*}\biggl(\Bigl(\bigcup_{w'\in G\ast H}w'\widehat{X}_0\Bigr)\bigcup\Bigr(\bigcup_{w'\in G\ast H}w'\widehat{Y_0}\Bigr)\biggr)=\biggl(\Bigl(\bigcup_{\substack{w'|_{\left|w\right|}=w \\ w'(\left|w'\right|)\in H}}w'\widehat{X}_0\Bigr)\bigcup\Bigl(\bigcup_{\substack{w'|_{\left|w\right|}=w \\ w'(\left|w'\right|)\in G}}w'\widehat{Y_0}\Bigr)\biggr)\bigcup\overline{W_k}\end{equation*} is covered by $\mathcal{U}'$.
\Par
Moreover, any $\overline{\alpha}\in \mathcal{A}$ also satisfies $d(\overline{\alpha},\alpha_k)=d(\overline{\alpha},w_kx_0)\leq\frac{\epsilon}{4}$ by similar calculations.  Since $w_kx_0\in w_k\widehat{X}_0$ or $w_kx_0\in w_k\widehat{Y_0}$, and $\left|w_k\right|=k$, then, like above, there is some $y\in\overline{W_k}$ such that $B_d(y,\epsilon)\in\mathcal{U}'$ and
\begin{eqnarray*}d(\overline{\alpha},y)\leq d(\overline{\alpha},\alpha_k)+d(\alpha_k,y)&=&d(\overline{\alpha},w_kx_0)+d(w_kx_0,y)
\\
&<&\frac{\epsilon}{4}+\frac{\epsilon}{2}<\epsilon\end{eqnarray*}

\noindent Therefore $\mathcal{U}'$ covers $\overline{W}$.
\Par
We finish the proof of the proposition by constructing a finite cover $\mathcal{U}$ of $\overline{W_k}$ by $\frac{\epsilon}{2}$-balls:
\Par
Begin with a finite cover $\mathcal{U}_0$ of $\widehat{X}_0\cup\widehat{Y_0}$ by $\frac{\epsilon}{2}$-balls; add in finitely many $\frac{\epsilon}{2}$-balls centered at points of $\partial X_0$ and $\partial Y_0$ to cover $\partial X_0$ and $\partial Y_0$.  Let $\frac{1}{2}>\delta_0>0$ be such that if $d(x,\partial X_0)<3\delta_0$, then $x$ lies in some element of $\mathcal{U}_0$ based at a point of $\partial X_0$, and if $d(y, \partial Y_0)<3\delta_0$, then $y$ lies in an element of $\mathcal{U}_0$ based at a point of $\partial Y_0$.
\Par
Choose $N>0$ such that $\frac{1}{2^{N}}<\delta_0\leq\frac{1}{2^{N-1}}$.
\Par
Now if $g\in G$ satisfies $d(gx_0,\partial X_0)<\delta_0$, then $r(g)\geq N$, so that $\text{diam}_d gY_0\leq\frac{1}{2^N}$.  Similarly, if $d(hy_0,\partial Y_0)<\delta_0$, then $\text{diam}_d hX_0\leq\frac{1}{2^N}$.
\Par
Let $A_1=\left\{g\in G\ |\ d(gx_0,\partial X_0)\geq\delta_0\right\}\cup\left\{ h\in H\ |\ d(hy_0,\partial Y_0)\geq\delta_0\right\}$.
\Par
Then $A_1$ is finite, and if $g\in G\backslash A_1$, then $d(gx_0,\partial X_0)<\delta_0$, and for any $x\in w \widehat{X}_0$ or $x\in w\widehat{Y_0}$ where $w|_1=g$, we have
\begin{equation*} d(x,gx_0)\leq\sum\limits_{n=N}^{\infty}\frac{1}{2^n}=\frac{1}{2^{N-1}}<2\delta_0\end{equation*}
so that
\begin{equation*}d(x,\partial X_0)\leq d(x,gx_0)+d(gx_0,\partial X_0)<2\delta_0+\delta_0=3\delta_0\end{equation*}
which implies that $x$ lies in some element of $\mathcal{U}_0$.
\Par
Therefore, $\mathcal{U}_0$ is a finite cover of $\widehat{X}_0\cup\widehat{Y_0}\cup\left(\bigcup_{\substack{w\in G\ast H \\ w|_1\notin A_1}}w\widehat{X}_0\right)\cup\left(\bigcup_{\substack{w\in G\ast H \\ w|_1\notin A_1}}w\widehat{Y_0}\right)$ by $\frac{\epsilon}{2}$-balls.
\Par
Now let $\mathcal{U}_1$ be a finite cover of $\left(\bigcup_{g\in A_1}g\widehat{Y}_0\right)\cup\left(\bigcup_{h\in A_1}h\widehat{X}_0\right)$ by $\frac{\epsilon}{2}$-balls.  Use a similar argument to the above to obtain, for each $g\in G\backslash A_1$ a finite subset $A_2^g\subseteq H$ such that if $h\in H\backslash A_2^g$, then $\left(\bigcup_{w|_2=gh}w\widehat{X}_0\right)\cup\left(\bigcup_{w|_2=gh}w\widehat{Y_0}\right)$ is contained in an element of $\mathcal{U}_1$.
\Par
Continue in this manner, letting $\mathcal{U}_m'=\bigcup_{i=0}^m\mathcal{U}_i$ for each $m=0,\ldots,k$.
\Par
Then $\mathcal{U}_m'$ covers $\overline{W_m}$ by finitely many $\frac{\epsilon}{2}$-balls for each $m=0,\ldots,k$, so that $\mathcal{U}:=\mathcal{U}_k'$ is a finite cover of $\overline{W_k}$ by $\frac{\epsilon}{2}$-balls, as desired.
\end{proof}

To see that $\overline{W}$ is an ANR, we will construct for each $\epsilon>0$ an ANR $Z_{\epsilon}\subseteq\overline{W}$ which $\epsilon$-dominates $\overline{W}$, and apply Theorem \ref{hanner}.
\Par
Given $\epsilon>0$, define
\begin{equation*}Z_{\epsilon}:=\widehat{X}_0\bigcup\widehat{Y}_0\bigcup\biggl(\bigcup_{\substack{|w|\leq k \\ w(1)\in A_1\cap H \\ w(i)\in A_{i}^{w|_{i-1}}}}w\widehat{X}_0\biggr)\bigcup\biggl(\bigcup_{\substack{|w|\leq k \\ w(1)\in A_1\cap G \\ w(i)\in A_{i}^{w|_{i-1}}}}w\widehat{Y}_0\biggr)\end{equation*}

where $k$, $A_1$, and $A^{w|_{i-1}}_i$ are defined as in the proof of Proposition \ref{Wbarcpt}.  Then $Z_{\epsilon}$ is a finite connected union of translates of $\widehat{X}_0$ and $\widehat{Y_0}$ with the property that
\begin{eqnarray*}\ &\text{If } w\widehat{X}_0\not\subseteq Z_{\epsilon},\text{ then diam}_d\left(\left(\bigcup_{w'|_{|w|}=w}w'\widehat{X}_0\right)\bigcup\left(\bigcup_{w'|_{|w|}=w}w'\widehat{Y_0}\right)\right)<\epsilon 
\\
\text{and} &\ 
\\
\ & \text{If } w\widehat{Y_0}\not\subseteq Z_{\epsilon},\text{ then diam}_d\left(\left(\bigcup_{w'|_{|w|}=w}w'\widehat{X}_0\right)\bigcup\left(\bigcup_{w'|_{|w|}=w}w'\widehat{Y_0}\right)\right)<\epsilon\end{eqnarray*}

Let $M_{\epsilon}$ denote the finite set of words in $G\ast H$ corresponding to the translates of $\widehat{X}_0$ and $\widehat{Y_0}$ in $Z_{\epsilon}$, and define a function $m:G\ast H\rightarrow \mathbb{N}\cup\left\{0\right\}$ by
\begin{equation*}m(w):=\max\left\{k\ |\ w|_k\in M_{\epsilon}\right\}\end{equation*}

Note that $m(w)=\left|w\right|$ if and only if $w\widehat{X}_0\subseteq Z_{\epsilon}$ (or $w\widehat{Y_0}\subseteq Z_{\epsilon}$).
\Par
Define maps $\phi:Z_{\epsilon}\rightarrow \overline{W}$ and $\psi:\overline{W}\rightarrow Z_{\epsilon}$ to be inclusion and ``projection'' maps, respectively.  By ``projection,'' we mean that $\psi|_{Z_{\epsilon}}\equiv \text{id}_{Z_{\epsilon}}$, and if, for example, $x\in w\widehat{X}_0$, where $w\notin M_{\epsilon}$, then $\psi(x)=w|_{m(w)+1}x_0\in Z_{\epsilon}$.

\begin{lemma}  \label{K}For any fixed $\epsilon>0$, let $Z_{\epsilon}$, $\phi$, and $\psi$ be defined as above.  Then there is a homotopy $K:\overline{W}\times[0,1]\rightarrow\overline{W}$ having the following properties:
\\
\tab (i)  $K$ is a $2\epsilon$-homotopy with $K_0\equiv\text{id}_{\overline{W}}$ and $K_1\equiv\phi\circ\psi$.
\\
\tab (ii)  $K_t(\overline{W}\backslash Z_{\epsilon})\cap\partial W=\emptyset$ for all $t>0$.\end{lemma}

Note that, by Lemma \ref{sdrzcompactification}, we may choose homotopies $F:\widehat{X}\times[0,1]\rightarrow\widehat{X}$ and $J:\widehat{Y}\times[0,1]\rightarrow \widehat{Y}$ and basepoints $x_0\in X$, $y_0\in Y$ such that
\begin{center}\begin{tabular}{c c} $F_0\equiv\text{id}_{\widehat{X}}\ \ $&$ \ \ J_0\equiv\text{id}_{\widehat{Y}}$
\\
\multicolumn{2}{c}{$F_t(\widehat{X})\cap\partial X=J_t(\widehat{Y})\cap\partial Y=\emptyset \text{ for all }t\in(0,1]$}
\\
$F_1(\widehat{X})=\left\{x_0\right\}\ \ $&$\ \ J_1(\widehat{Y})=\left\{y_0\right\}$
\\
$F(x_0,t)=x_0\text{ for all }t\in[0,1]$ & $J(y_0,t)=y_0\text{ for all }t\in[0,1]$\end{tabular}\end{center}

Observe also that, by Lemma \ref{slowedsdr}, we may assume that $F$ and $J$ satisfy (in addition to being $\mathcal{Z}$-set homotopies which are strong deformation retractions)
\begin{center}$ F_t|_{\widehat{X}\backslash B_{\rho}(\partial X,\frac{1}{2^k})}\equiv \text{id}_{\widehat{X}\backslash B_{\rho}(\partial X,\frac{1}{2^k})}\text{ if } t\in[0,\frac{1}{2^k}]$\end{center}

and

\begin{center}$J_t|_{\widehat{Y}\backslash B_{\tau}(\partial Y,\frac{1}{2^k})}\equiv \text{id}_{\widehat{Y}\backslash B_{\tau}(\partial Y,\frac{1}{2^k})}\text{ if } t\in[0,\frac{1}{2^k}]$\end{center}

These homotopies are used to construct $K$ and also to prove Proposition \ref{freeprodZset}.  We refer the reader to the end of the chapter for the proof of Lemma \ref{K}.

\begin{proposition}\label{WbarANR} $\overline{W}$ is an ANR.\end{proposition}
\begin{proof} By Theorem \ref{hanner}, it suffices to show that for every $\epsilon>0$, there is an ANR which $2\epsilon$-dominates $\overline{W}$.
\Par
Fix $\epsilon>0$, and let $Z_{\epsilon}$ be defined as above.
\Par
As a subspace of $\overline{W}$, it is clear that $Z_{\epsilon}$ is metrizable.  That $Z_{\epsilon}$ is an ANR follows from the fact that translates of $\widehat{X}_0$ and $\widehat{Y}_0$ are glued together along at most one point, and the inductive application of the following theorem:

\begin{theorem}(See Section VI.1 of \cite{hu})  If $A$, $B$, and $C$ are ANR's, with $A\subseteq B$, and if the adjunction space $Z$ of the map $g:A\rightarrow C$ is metrizable, then $Z$ is an ANR.\end{theorem}

In this situation, we take as $B$ a finite connected union of translates of $\widehat{X}_0$ and $\widehat{Y}_0$, as $C$ another translate of $\widehat{X}_0$ or $\widehat{Y}_0$ which is to be connected to $B$, and as $A$ the single point in $B$ at which $C$ is to be attached.  The map $g:A\rightarrow C$ is the obvious one, and the adjunction space $Z$ is the disjoint union of $B$ and $C$ modulo the equivalence relation which identifies the single point in $A$ to its image under $g$.  It is clear that the spaces $A$, $B$, and $C$ are ANR's and that $Z$ is metrizable, so the theorem applies.
\Par
Lemma \ref{K} implies that $Z_{\epsilon}$ $2\epsilon$-dominates $\overline{W}$.  Therefore, Theorem \ref{hanner} applies, and $\overline{W}$ is an ANR.\end{proof}

\begin{corollary}\label{WbarER}$\overline{W}$ is an ER.\end{corollary}
\begin{proof}  By Proposition \ref{WbarANR} and Fact \ref{ARER}, it suffices to show that $\overline{W}$ is finite dimensional and contractible.
\Par
Fix $\epsilon>0$.  Lemma \ref{K} shows that $\overline{W}$ is $\epsilon$-dominated by a compact metric space $Z_{\frac{\epsilon}{2}}$ whose dimension is bounded above by the maximum of the dimensions of $\widehat{X}$ and $\widehat{Y}$.  We claim that the map $K_1:\overline{W}\rightarrow Z_{\frac{\epsilon}{2}}$, where $K:\overline{W}\times[0,1]\rightarrow\overline{W}$ is the $\frac{\epsilon}{2}$-homotopy given by Lemma \ref{K}, is an $\epsilon$-mapping:
\Par
For each $z\in Z_{\frac{\epsilon}{2}}$, either $K_1^{-1}(\left\{z\right\})=\left\{z\right\}$ (in which case, it is certainly true that $\text{diam}_dK_1^{-1}(\left\{z\right\})=0<\epsilon$), or $z=wx_0$ for some $w\in G\ast H$ satisfying $m(w)=\left|w\right|-1$.  In this second case, we have $K_1^{-1}(\left\{z\right\})= \mathcal{B}_w$, where
\begin{equation*}\mathcal{B}_w:=\biggl(\bigcup_{w'|_{|w|}=w}w'\widehat{X}_0\biggr)\bigcup\biggl(\bigcup_{w'|_{|w|}=w}w'\widehat{Y_0}\biggr)\bigcup\biggl\{\left\{w_ix_0\right\}_{i=1}^{\infty}\in\mathcal{A}\ |\ (w_i)|_{|w|}=w\biggr\},\end{equation*}

consists of all the branches coming off of (and including) $w\widehat{X}_0$ (or $w\widehat{Y}_0$).  Then $\text{diam}_dK_1^{-1}(\left\{z\right\})=\text{diam}_d\mathcal{B}_w\leq\frac{\epsilon}{2}<\epsilon$ by definition of $Z_{\frac{\epsilon}{2}}$.
\Par
Hence we have, for each $\epsilon>0$, an $\epsilon$-mapping of $\overline{W}$ to a compact metric space $Z_{\frac{\epsilon}{2}}$ with $\dim Z_{\frac{\epsilon}{2}}\leq\max{\dim\widehat{X},\dim\widehat{Y}}$.  Therefore Theorem \ref{engelking} applies, and $\overline{W}$ is finite-dimensional.
\Par
Moreover, $\overline{W}$ is contractible, since it is homotopy equivalent to the contractible $Z_{\epsilon}$.
\Par
Therefore, $\overline{W}$ is an ER.
\end{proof}
\begin{proposition}\label{freeprodZset}  $\partial W$ is a $\mathcal{Z}$-set in $\overline{W}$.\end{proposition}

\begin{proof}  We must construct a homotopy $P:\overline{W}\times[0,1]\rightarrow \overline{W}$ with the property that $P_0\equiv \text{id}_{\overline{W}}$ and $P_t(\overline{W})\subseteq W$ for all $t>0$.
\Par
Recall that, given any $\epsilon>0$ and a space $Z_{\epsilon}\subseteq\overline{W}$ with the property that branches outside of $Z_{\epsilon}$ have diameter smaller than $\epsilon$, Lemma \ref{K} gives a $2\epsilon$-homotopy $K:\overline{W}\times[0,1]\rightarrow\overline{W}$ which satisfies $K_t(\overline{W}\backslash\partial W)=\emptyset$ for any $t>0$.  This homotopy, of course, depends on both $\epsilon$ and the choice of the space $Z_{\epsilon}$.
\Par
To build the $\mathcal{Z}$-set homotopy $P$, we first fix $\epsilon=1$ and $Z_{\epsilon}=\widehat{X}_0\cup\widehat{Y}_0$.  Then we let $K$ be the homotopy given by Lemma \ref{K} with these choices in place.  Now we have $K_0\equiv\text{id}_{\overline{W}}$ and $K_1(\overline{W})\subseteq Z_{\epsilon}=\widehat{X}_0\cup\widehat{Y}_0$.
\Par
Observe that for each $x\in\overline{W}$, either $x\in\widehat{X}_0\cup\widehat{Y_0}$ or there exists a unique $g\in G$ (or $h\in H$) such that $x\in \mathcal{B}_g$ (or $x\in \mathcal{B}_h$), where 
\begin{equation*}\mathcal{B}_g:=\biggl(\bigcup_{w'|_1=g}w'\widehat{X}_0\biggr)\bigcup\biggl(\bigcup_{w'|_1=g}w'\widehat{Y_0}\biggr)\bigcup\biggl\{\left\{w_ix_0\right\}_{i=1}^{\infty}\in\mathcal{A}\ |\ (w_i)|_1=g\biggr\}\end{equation*}
and, similarly,
\begin{equation*}\mathcal{B}_h:=\biggl(\bigcup_{w'|_1=h}w'\widehat{X}_0\biggr)\bigcup\biggl(\bigcup_{w'|_1=h}w'\widehat{Y_0}\biggr)\bigcup\biggl\{\left\{w_ix_0\right\}_{i=1}^{\infty}\in\mathcal{A}\ |\ (w_i)|_1=h\biggr\},\end{equation*}

Now we define
\begin{equation*} P(x,t):=\left\{\begin{array}{l l}F(x,t) & \text{ for any } t\in[0,1] \text{ if } x\in\widehat{X}_0
\\
J(x,t) & \text{ for any } t\in[0,1] \text{ if } x\in\widehat{Y}_0
\\
K(x,2^{r(h)}\cdot t) & \text{ if } x\in\mathcal{B}_h \text { and } t\in[0,2^{r(h)}]
\\
K(x,2^{r(g)}\cdot t) & \text{ if } x\in\mathcal{B}_g \text { and } t\in[0,2^{r(g)}]
\\
F(x,t) & \text{ if } x\in\mathcal{B}_h \text{ and } t\in[2^{r(h)},1]
\\
J(x,t) & \text{ if } x\in\mathcal{B}_g \text{ and } t\in[2^{r(g)},1]\end{array}\right.\end{equation*}

where $r:G\cup H\rightarrow [0,1]$ is as defined at the beginning of the chapter.
\Par
That $P$ is continuous follows from the pasting lemma for continuous functions, and the properties of $F$, $J$, and $K$ imply that $P$ has the desired attributes.
\end{proof}

\begin{theorem}\label{zstructurethmfreeprod}If both $G$ and $H$ admit $\mathcal{Z}$-structures, then so does $G\ast H$.\end{theorem}
\begin{proof}First, it is clear that $G\ast H$ acts cocompactly on $W$, since if the translates of $C$ and $D$ under the actions of $G$ and $H$ cover $X$ and $Y$, respectively, then the translates of $C\cup D$ under the action of $G\ast H$ cover $W$.  Moreover, given a translate $Z$ of $X_0$ or $Y_0$, each element of $G\ast H$ either fixes $Z$ (in which case, the action is proper) or moves $Z$ completely off itself, so that the action of $G\ast H$ on $W$ is also proper.  Therefore $(\overline{W},\partial W)$ satisfies condition (3) of Definition \ref{zstructure}.
\Par
Propositions \ref{Wbarcpt} and \ref{freeprodZset}, and Corollary \ref{WbarER} prove that conditions (1) and (2) are satisfied by the pair $(\overline{W},\partial W)$.
\Par
It remains only to show that $\overline{W}$ satisfies the null condition with respect to the action of $G\ast H$ on $W$.  This follows directly from the facts each of the original actions have this property and that for any $\epsilon>0$, there are only finitely many translates of $\widehat{X}_0$ and $\widehat{Y}_0$ with $d$-diameter more than $\epsilon$.  Hence condition (4) of Definition \ref{zstructure} is also satisfied.
\Par
Therefore $(\overline{W},\partial W)$ is a $\mathcal{Z}$-structure on $G\ast H$.
\end{proof}

\begin{theorem} \label{ezstructurefreeprod} If $G$ and $H$ each admit $\mathcal{EZ}$-structures, then so does $G\ast H$.\end{theorem}

\begin{proof}  We show that $(\overline{W},\partial W)$, as defined in the proof of Theorem \ref{zstructurethmfreeprod}, satisfies the axioms for an $\mathcal{EZ}$-structure.  By Theorem \ref{zstructurethmfreeprod}, it remains only to show that the action of $G\ast H$ on $W$ extends to an action on $\overline{W}$.
\Par
Recall that a point $\alpha\in\partial W$ has one of three types:  (i)  $\alpha\in w\partial X_0$, (ii)  $\alpha\in w\partial Y_0$, or (iii)  $\alpha\in\mathcal{A}$, where

\begin{equation*}\mathcal{A}=\Bigl\{\overline{\alpha}\ |\ \overline{\alpha}=\left\{\alpha_i\right\}_{i=1}^{\infty}\ ,\ \alpha_i=w_ix_0\ , \ w_i\in G\ast H\ , \ \left|w_i\right|=i\ ,\ w_{i+1}|_i=w_i\ \forall i\in\mathbb{N}\Bigr\}\end{equation*}

Under the assumption that the actions of $G$ and $H$ extend to actions on $\widehat{X}_0$ and $\widehat{Y_0}$, the action of $G\ast H$ on $W$ extends to points of $\partial W$ having type (i) and (ii) in the obvious way.
\Par
For a point $\alpha=\left\{\alpha_i\right\}_{i=1}^{\infty}\in\mathcal{A}$, let $w\cdot\alpha:=\left\{w\cdot\alpha_i\right\}_{i=1}^{\infty}\in\mathcal{A}$, and the theorem is proved.
\end{proof}

We conclude the chapter with the proof of Lemma \ref{K}:
\Par
\textit{Proof of Lemma \ref{K}.}  For a given word $w$, $j(w):=|w|-m(w)$ indicates in some sense how ``far'' $w\widehat{X}_0$ (or $w\widehat{Y_0}$) is projected by $\psi$.
\Par
Recall the function $r:G\cup H\rightarrow \mathbb{N}$ defined earlier in the chapter by
\\
\tab $r(g)=n\ \Longleftrightarrow \ gx_0\in B_{\rho}(\partial X,\frac{1}{2^{n-1}})\backslash B_{\rho}(\partial{X},\frac{1}{2^n})$
\\
\tab $r(h)=n\ \Longleftrightarrow\  hy_0\in B_{\tau}(\partial Y,\frac{1}{2^{n-1}})\backslash B_{\tau}(\partial{Y},\frac{1}{2^n})$.
\Par

Recall that $F:\widehat{X}\times[0,1]\rightarrow\widehat{X}$ and $J:\widehat{Y}\times[0,1]\rightarrow\widehat{X}$ satisfy
\begin{center} $F_t|_{\widehat{X}\backslash B_{\rho}(\partial X,\frac{1}{2^k})}\equiv \text{id}_{\widehat{X}\backslash B_{\rho}(\partial X,\frac{1}{2^k})}\text{ if } t\in[0,\frac{1}{2^k}]$\end{center}

and

\begin{center} $J_t|_{\widehat{Y}\backslash B_{\tau}(\partial Y,\frac{1}{2^k})}\equiv \text{id}_{\widehat{Y}\backslash B_{\tau}(\partial Y,\frac{1}{2^k})}\text{ if } t\in[0,\frac{1}{2^k}]$\end{center}

This implies that each point in $X$ [resp. $Y$] remains fixed under the homotopy $F$ [resp. $J$] on a pre-determined interval around $t=0$.  In particular, for any $g\in G$, we have $F\left(\left\{gx_0\right\}\times\left[0,\frac{1}{2^{r(g)}}\right]\right)=\left\{gx_0\right\}$, and similarly for $h\in H$.  We use this fact to define a homotopy $K:\overline{W}\times[0,1]\rightarrow\overline{W}$ from $\text{id}_{\overline{W}}$ to $\phi\circ\psi$ by concatenating translates of $F$ and $J$ in such a way that two translated homotopies agree when they intersect at a gluing point and the entire ``branch'' of $\overline{W}$ coming off of any given gluing point is pulled in by $K$ during the time that the gluing point remains fixed.  This systematic concatenation of the $\mathcal{Z}$-set homotopies allows the definition of $K$ to be extended to points of $\mathcal{A}$.  Here we give an inductive definition for $K$, and, in hopes of simplifying the ideas used, we give a figure below illustrating an example of its execution on a specific branch of $\overline{W}$.
\Par
We first define $K^0:Z_{\epsilon}\times[0,1]\rightarrow \overline{W}$ by $K^0(z,t):=z$ for all $(z,t)\in Z_{\epsilon}\times[0,1]$.
\Par
To define $K$ on the rest of $\overline{W}\backslash\mathcal{A}$, first note that $z\in\overline{W}\backslash \left(\mathcal{A}\cup Z_{\epsilon}\right)\ \Longrightarrow z\in w\widehat{X}_0$ or $z\in w\widehat{Y_0}$, where $j(w)\in\mathbb{N}$.
\Par
To each $w\notin M_{\epsilon}$ with $j(w)\geq 2$, we associate a number $t(w)=\prod\limits_{i=1}^{j(w)-1}\frac{1}{2^{r(w(|w|-(i-1)))}}\in(0,1)$.  (The entire branch coming off the gluing point $wx_0$ will be pulled in by $K$ to $wx_0$ on the interval $[0,t(w)]$.)
\Par
Define $\displaystyle Q_n:=\Bigl(\bigcup_{j(w)=n}w\widehat{X}_0\Bigr)\bigcup\Bigl(\bigcup_{j(w)=n}w\widehat{Y}_0\Bigr)$ and $\displaystyle Q^n:=Z_{\epsilon}\cup\left(\bigcup_{i=1}^nQ_n\right)$ for each $n\in\mathbb{N}$.
\Par
We will use induction on $n$ to define a homotopy $K^n:Q^n\times[0,1]\rightarrow\overline{W}$ and set $K:=\bigcup_{n=0}^{\infty}K^n:\overline{W}\backslash\mathcal{A}$.  Then we will extend $K$ to $\mathcal{A}$ by taking appropriate limits.
\Par
First let $K_1:Q_1\times[0,1]\rightarrow\overline{W}$ be defined by

\begin{equation*} K_1(z,t):=\left\{\begin{array}{l l}wF(z,t) & \text{if } z\in w\widehat{X}_0 \text{ with }j(w)=1
\\
wJ(z,t) & \text{if }z\in w\widehat{Y}_0 \text{ with } j(w)=1\end{array}\right.\end{equation*}

and set $K^1:=K^0\cup K_1:Q^1\times[0,1]\rightarrow\overline{W}$.
\Par
We show that $K^0$ and $K_1$ agree on the intersection $Z_{\epsilon}\cap Q_1$ and conclude that $K^1$ is continuous:
\Par
Note that $Z_{\epsilon}\cap Q_1=\left\{wx_0\ |\ j(w)=1\right\}$.  For any such $wx_0\in Z_{\epsilon}\cap Q_1$, either $w(|w|)\in G$ or $w(|w|)\in H$; assume without loss of generality that $w(|w|)\in G$.  Then $wx_0=wy_0\in Z_{\epsilon}\cap w\widehat{Y}_0$, and $K_1(wx_0,t)=K_1(wy_0,t)=wJ(wy_0,t)=wy_0$ for all $t\in[0,1]$ since $J$ is a strong deformation retraction.  Thus $K_1(wx_0,t)=K^0(wx_0,t)$ for all $t\in[0,1]$.
\Par
Next we define $K_2:Q_2\times[0,1]\rightarrow\overline{W}$ by

\begin{equation*} K_2(z,t):=\left\{\begin{array}{l l}wF(z,\frac{t}{t(w)}) & \text{if } z\in w\widehat{X}_0 \text{ with }j(w)=2\text{ and }t\in[0,t(w)]
\\
K^1(wy_0,t)& \text{if } z\in w\widehat{X}_0 \text{ with }j(w)=2\text{ and }t\in[t(w),1]
\\
wJ(z,t) & \text{if }z\in w\widehat{Y}_0 \text{ with } j(w)=2
\\
K^1(wx_0,t)& \text{if } z\in w\widehat{Y}_0 \text{ with }j(w)=2\text{ and }t\in[t(w),1]\end{array}\right.\end{equation*}

and set $K^2:=K^1\cup K_2:Q^2\times[0,1]\rightarrow\overline{W}$.
\Par
We again show that $K^1$ and $K_2$ agree on the intersection $Q^1\cap Q_2=\left\{wx_0\ |\ j(w)=2\right\}$ to conclude that $K^2$ is continuous:
\Par
Given $wx_0$ with $j(w)=2$, assume without loss of generality that $w(|w|)=g\in G$.  Then $wx_0\in w|_{|w|-1}\widehat{X}_0\cap w\widehat{Y}_0\subseteq Q^1\cap Q_2$.
\Par
Then $K^1(wx_0,t)=K_1(wx_0,t)=w|_{|w|-1}F(wx_0,t)$ for all $t\in[0,1]$.  Since $F(gx_0,t)=gx_0$ for all $t\leq\frac{1}{2^{r(g)}}$, then $K^1(wx_0,t)=w|_{|w|-1}F(wx_0,t)=wx_0$ for all $t\leq \frac{1}{2^{r(g)}}=t(w)$.  On the other hand, $K_2(wx_0,t)=wJ(wx_0, \frac{1}{t(w)})=wx_0$ for all $t\leq t(w)$.  Moreover, $K^1(wx_0,t)=K_1(wx_0,t)=K_2(wx_0,t)$ for all $t\geq t(w)$, by definition.  Therefore $K^1$ and $K_2$ agree on $\left\{wx_0\right\}\times[0,1]$ for every $wx_0\in Q^1\cap Q_2$, so $K^2$ is continuous.
\Par
Lastly, we observe that if $j(w)=3$, then $K^2(wx_0,t)=wx_0$ for all $t\leq t(w)$:
\Par
Suppose $j(w)=3$; then $wx_0\in w|_{|w|-1}\widehat{X}_0$ (if $w(|w|)\in G$) or $wx_0\in w|_{|w|-1}\widehat{Y}_0$ (if $w(|w|)\in H$).  Assume, without loss of generality, that $w(|w|)=g\in G$.  Then $K^2(wx_0,t)=K_2(wx_0,t)=w|_{|w|-1}F(wx_0,\frac{t}{t(w|_{|w|-1})}$ for all $t\leq t(w|_{|w|-1})$.  Since $F(gx_0,t)=gx_0$ for all $t\leq \frac{1}{2^{r(g)}}$, then $w|_{|w|-1}F(wx_0,\frac{t}{t(w|_{|w|-1})})=wx_0$ whenever $\frac{t}{t(w|_{|w|-1})}\leq\frac{1}{2^{r(g)}}$, which holds whenever $t\leq t(w|_{|w|-1})\cdot\frac{1}{2^{r(g)}}=t(w)$.
\Par
Continue inductively:  Suppose $K^{n-1}:Q^{n-1}\times[0,1]\rightarrow\overline{W}$ is continuous and satisfies, for each $w\in G\ast H$ with $j(w)\leq n$,
\begin{equation*}K^{n-1}(wx_0,t)=wx_0\text{ for all }t\leq t(w)\end{equation*}

Define $K_n:Q_n\times[0,1]\rightarrow\overline{W}$ by

\begin{equation*} K_n(z,t):=\left\{\begin{array}{l l}wF(z,\frac{t}{t(w)}) & \text{if } z\in w\widehat{X}_0 \text{ with }j(w)=n\text{ and }t\in[0,t(w)]
\\
K^{n-1}(wy_0,t)& \text{if } z\in w\widehat{X}_0 \text{ with }j(w)=n\text{ and }t\in[t(w),1]
\\
wJ(z,t) & \text{if }z\in w\widehat{Y}_0 \text{ with } j(w)=n
\\
K^{n-1}(wx_0,t)& \text{if } z\in w\widehat{Y}_0 \text{ with }j(w)=n\text{ and }t\in[t(w),1]\end{array}\right.\end{equation*}

An identical argument to the above, showing that $K^2$ is continuous, shows that $K^{n}:=K^{n-1}\cup K_n:Q^n\times[0,1]\rightarrow\overline{W}$ is continuous.  Moreover, if $j(w)=n+1$, then, assuming $w(|w|)=g\in G$ and setting $w':=w|_{|w|-1}$, we have $wx_0\in w'\widehat{X}_0$, and $K^n(wx_0,t)=K_n(wx_0,t)=w'F(wx_0,\frac{t}{t(w')})$ for any $t\leq t(w')$.  But $F(gx_0,t)=gx_0$ for all $t\leq\frac{1}{2^{r(g)}}$, so that $w'F(wx_0,\frac{t}{t(w')})=wx_0$ whenever $\frac{t}{t(w')}\leq\frac{1}{2^{r(g)}}$, which occurs for any $t\leq t(w')\cdot\frac{1}{2^{r(g)}}=t(w)$.
\Par
\textit{Example:}   (See Figure \ref{fig:K}.)  Suppose $j(w)=4$, and $w=w'ghg'$, where $w'x_0\in Z_{\epsilon}$, $r(g')=1$, $r(h)=3$, and $r(g)=2$.  Then we have
\begin{eqnarray*}  t(w) & = & \frac{1}{2^{r(g')}}\cdot\frac{1}{2^{r(h)}}\cdot\frac{1}{2^{r(g)}}= \frac{1}{2}\cdot\frac{1}{8}\cdot\frac{1}{4}=\frac{1}{64}
\\
t(w'gh) & = & \frac{1}{2^{r(h)}}\cdot\frac{1}{2^{r(g)}}=\frac{1}{8}\cdot\frac{1}{4}=\frac{1}{32}
\\
t(w'g) & = & \frac{1}{2^{r(g)}}=\frac{1}{4}\end{eqnarray*}

\begin{figure}[!h]
\includegraphics{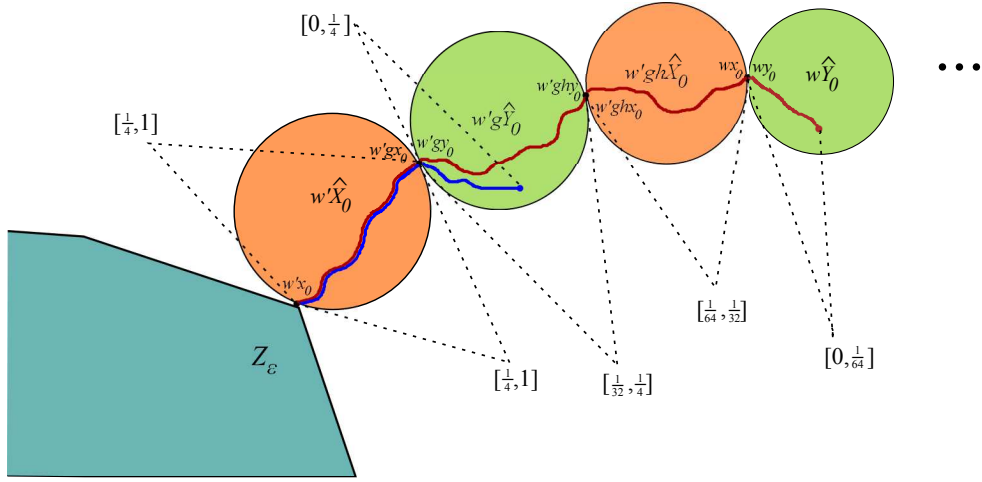}
\caption{The homotopy $K$ on a single branch of $\overline{W}$}
\label{fig:K}\end{figure}

We conclude the construction of $K$ by extending to $\mathcal{A}$:
\Par
Suppose $z=\left\{w_ix_0\right\}_{i=1}^{\infty}\in\mathcal{A}$.  Then $w_ix_0\rightarrow z$ as $i\rightarrow\infty$, and we set, for each $t\in[0,1]$,
\begin{equation*}K(z,t):=\lim\limits_{i\rightarrow\infty}K(w_ix_0,t)\end{equation*}

\noindent Continuity of $K:\overline{W}\times[0,1]\rightarrow\overline{W}$ is implied by the induction argument above and the following simple facts:
\\
\tab $\bullet$  $K\left(w\widehat{X}_0\times\left[0,t(w)\right]\right)\subseteq w\widehat{X}_0$ for any $w\in G\ast H$
\\
\tab $\bullet$  $K\left(w\widehat{Y_0}\times\left[0,t(w)\right]\right)\subseteq w\widehat{Y_0}$ for any $w\in G\ast H$
\\
\tab $\bullet$  $\displaystyle K\left(w\widehat{X}_0\times\left[t(w|_{|w|-(i-1)}),t(w|_{|w|-i})\right]\right)\subseteq$\begin{math}\left\{\begin{array}{l l}w|_{|w|-i}\widehat{Y_0} & \text{ if } 1\leq i\leq j(w)-2 \text{ is odd}
\\
w|_{|w|-i}\widehat{X}_0 & \text{ if } 1\leq i\leq j(w)-2 \text{ is even}\end{array}\right.\end{math}
\\
\tab $\bullet$  $\displaystyle K\left(w\widehat{Y_0}\times\left[t(w|_{|w|-(i-1)}),t(w|_{|w|-i})\right]\right)\subseteq$\begin{math}\left\{\begin{array}{l l}w|_{|w|-i}\widehat{X}_0 & \text{ if } 1\leq i\leq j(w)-2 \text{ is odd}
\\
w|_{|w|-i}\widehat{Y_0} & \text{ if } 1\leq i\leq j(w)-2 \text{ is even}\end{array}\right.\end{math}
\\
\tab $\bullet$  If $w$ satisfies $|w|>m(w)$ (i.e. $j(w)>0$), and
\begin{equation*}\mathcal{B}_w:=\biggl(\bigcup_{w'|_{|w|}=w}w'\widehat{X}_0\biggr)\bigcup\biggl(\bigcup_{w'|_{|w|}=w}w'\widehat{Y_0}\biggr)\bigcup\biggl\{\left\{w_ix_0\right\}_{i=1}^{\infty}\in\mathcal{A}\ |\ (w_i)|_{|w|}=w\biggr\},\end{equation*}

consists of all the branches coming off of (and including) $w\widehat{X}_0$ (or $w\widehat{Y}_0$), then
 \begin{equation*}K\left(\mathcal{B}_w\times\displaystyle\left[0,t(w)\right]\right)\subseteq\mathcal{B}_w\end{equation*}
and
\begin{equation*}K\left(\mathcal{B}_w\times\displaystyle[0,1]\right)\subseteq\mathcal{B}_{w|_{m(w)+1}}\end{equation*}

The properties of $Z_{\epsilon}$ imply that $\text{diam}_d\mathcal{B}_w\leq\epsilon<2\epsilon$ whenever $j(w)>0$, so that $\text{diam}_d\left(K\left(\left\{z\right\}\times[0,1]\right)\right)<2\epsilon$ for any $z\in\overline{W}$.
\Par
Hence, $K$ is a $2\epsilon$-homotopy between $\text{id}_{\overline{W}}$ and $\phi\circ\psi$.
\Par
Moreover, $K_t(\overline{W}\backslash Z_{\epsilon})\cap\partial W=\emptyset$ for all $t>0$ since $F$ and $J$ are $\mathcal{Z}$-set homotopies, and due to the limit definition of $K$ at points of $\mathcal{A}$.
\begin{flushright}$\blacksquare$\end{flushright}
\section{$\mathcal{Z}$- and $\mathcal{EZ}$-Structures on Direct Products of Groups}
\pagestyle{headings} \thispagestyle{headings}
\vspace{.2in}

\begin{fact}  \label{zcompprod}The product $\widehat{X}\times\widehat{Y}$ of $\mathcal{Z}$-compactifications $\widehat{X}$ and $\widehat{Y}$ of $X$ and $Y$, respectively, is a $\mathcal{Z}$-compactification of $X\times Y$.\end{fact}
\begin{proof}  It is a standard fact that a product of ANR's is an ANR.  Thus $\widehat{X}\times\widehat{Y}$ is an ANR.
\Par
Suppose $F:\widehat{X}\times[0,1]\rightarrow\widehat{X}$ and $G:\widehat{Y}\times [0,1]\rightarrow\widehat{Y}$ are $\mathcal{Z}$-set homotopies.  Define $H:\widehat{X}\times\widehat{Y}\times[0,1]\rightarrow \widehat{X}\times\widehat{Y}$ by $H(\widehat{x},\widehat{y},t):=(F(\widehat{x},t),G(\widehat{y},t))$
for all $(\widehat{x},\widehat{y},t)\in\widehat{X}\times\widehat{Y}\times[0,1]$.
\Par
Since $F_0\equiv\text{id}_{\widehat{X}}$ and $G_0\equiv\text{id}_{\widehat{Y}}$, then $H_0\equiv\text{id}_{\widehat{X}\times\widehat{Y}}$.
\Par
Moreover, since $F_t(\widehat{X})\subseteq X$ and $G_t(\widehat{Y})\subseteq Y$ for any $t>0$, then $H_t(\widehat{X}\times\widehat{Y})\subseteq X\times Y$ whenever $t>0$.
\Par
Hence $H_t(\widehat{X}\times\widehat{Y})\cap\left[(\partial X\times Y)\cup(X\times\partial Y)\right]=\emptyset$ for any $t>0$.
\Par
Therefore, $\left[(\partial X\times Y)\cup(X\times\partial Y)\right]$ is a $\mathcal{Z}$-set in $\widehat{X}\times\widehat{Y}$.
\end{proof}

Unfortunately, the analogous result for $\mathcal{Z}$-structures does not hold:  Suppose $(\widehat{X},\partial X)$ and $(\widehat{Y},\partial Y)$ are $\mathcal{Z}$-structures on $G$ and $H$, respectively.  Although, by Fact \ref{zcompprod}, $\widehat{X}\times\widehat{Y}$ is a $\mathcal{Z}$-compactification of $X\times Y$, the space $\widehat{X}\times\widehat{Y}$ does not, in general, satisfy the null condition with respect to the action of $G\times H$ on $X\times Y$.

\begin{example} Let $\widehat{\mathbb{R}}$ denote the $\mathcal{Z}$-compactification of the real line $\mathbb{R}$ by two points, and consider the $\mathcal{Z}$-compactification $\widehat{\mathbb{R}}\times\widehat{\mathbb{R}}$ of $\mathbb{R}^2$, the Euclidean plane.  
\Par
Observe that $\widehat{\mathbb{R}}$ is a $\mathcal{Z}$-structure on $\mathbb{Z}$, but, with the product topology, $\widehat{\mathbb{R}}\times\widehat{\mathbb{R}}$ is not a $\mathcal{Z}$-structure on $\mathbb{Z}\times\mathbb{Z}$:
\Par
Write $\widehat{\mathbb{R}}:=\left\{\alpha\right\}\cup(-\infty,\infty)\cup\left\{\beta\right\}$.  The set $\mathcal{B}=\left\{(a,b)\ |\ a,b\in\mathbb{R}\right\}\cup\left\{\left\{\alpha\right\}\cup(-\infty,a)\ |\ a\in \mathbb{R}\right\}\cup\left\{(b,\infty)\cup\left\{\beta\right\}\ |\ b\in\mathbb{R}\right\}$ is a basis for the topology on $\widehat{\mathbb{R}}$.
\Par
Now $\widehat{\mathbb{R}^2}:=\widehat{\mathbb{R}}_1\times\widehat{\mathbb{R}}_2$, with the product topology.
\Par
Note in Figure \ref{prodnbhds} some typical neighborhoods of boundary points in $\widehat{\mathbb{R}^2}$.
\Par

\begin{figure}[!h]\begin{minipage}{3.2in}\begin{center}\begin{tikzpicture}[scale=1.15]
\fill[color=red!30] (-1.5,-1.5) rectangle (-1.001,1.5);
\fill[color=blue!30] (-0.5,.75) rectangle (0.25,1.5);
\fill[color=green!30] (-0.25,-1.5) rectangle (1.5,-0.8);
\draw [-] (-1.5,1.5) -- (1.5,1.5) -- (1.5,-1.5) -- (-1.5,-1.5) --(-1.5,1.5);
\draw [-] (-1.5,0) -- (1.5,0);
\draw [-] (0,-1.5) -- (0,1.5);
\draw [dashed] (-1,1.5)--(-1,-1.5);
\draw [dashed] (-0.5,1.5) -- (-0.5,.75) -- (0.25,0.75) -- (0.25,1.5);
\draw [dashed] (-0.25,-1.5) -- (-0.25,-0.8) -- (1.5,-0.8);

\fill (-1.5,0) circle (1.1pt) node [left=1pt] {$(\alpha_1,0)$};
\fill (1.5,0) circle (1.1pt) node [right=1pt] {$(\beta_1,0)$};
\fill (0,-1.5) circle (1.1pt) node [below=1pt] {$(0,\alpha_2)$};
\fill (0,1.5) circle (1.1pt) node [above=1pt] {$(0,\beta_2)$};
\fill (-1.5,1.5) circle (1.1pt) node [above=1pt] {$(\alpha_1,\beta_2)$};
\fill (1.5,1.5) circle (1.1pt) node [above=1pt] {$(\beta_1,\beta_2)$};
\fill (-1.5,-1.5) circle (1.1pt) node [below=1pt] {$(\alpha_1,\alpha_2)$};
\fill (1.5,-1.5) circle (1.1pt) node [below=1pt] {$(\beta_1,\alpha_2)$};
\draw [gray,decorate,decoration={brace,amplitude=5pt},xshift=-17pt]
   (-1.9,-1.5)  -- (-1.9,1.5) 
   node [black,midway,left=4pt,xshift=-2pt] {$\widehat{\mathbb{R}}_2$};
\draw [gray,decorate,decoration={brace,amplitude=5pt},yshift=-7pt]
   (1.5,-1.8)  -- (-1.5,-1.8) 
   node [black,midway,below=4pt,xshift=-2pt] {$\widehat{\mathbb{R}}_1$};

\end{tikzpicture}\end{center}
\label{prodnbhds}
\caption{Neighborhoods of boundary points in $\widehat{\mathbb{R}^2}$}\end{minipage}
\begin{minipage}{3.2in}\begin{center}\begin{tikzpicture}[scale=1]
\draw [-] (-1.5,1.5) -- (1.5,1.5) -- (1.5,-1.5) -- (-1.5,-1.5) --(-1.5,1.5);
\draw [-] (-1.5,0) -- (1.5,0);
\draw [-] (0,-1.5) -- (0,1.5);
\fill (-1.5,0) circle (1.1pt) node [left=1pt] {$(\alpha_1,0)$};
\fill (1.5,0) circle (1.1pt) node [right=1pt] {$(\beta_1,0)$};
\fill (0,-1.5) circle (1.1pt) node [below=1pt] {$(0,\alpha_2)$};
\fill (0,1.5) circle (1.1pt) node [above=1pt] {$(0,\beta_2)$};
\fill (-1.5,1.5) circle (1.1pt) node [above=1pt] {$(\alpha_1,\beta_2)$};
\fill (1.5,1.5) circle (1.1pt) node [above=1pt] {$(\beta_1,\beta_2)$};
\fill (-1.5,-1.5) circle (1.1pt) node [below=1pt] {$(\alpha_1,\alpha_2)$};
\fill (1.5,-1.5) circle (1.1pt) node [below=1pt] {$(\beta_1,\alpha_2)$};
\draw [gray,decorate,decoration={brace,amplitude=5pt},xshift=-25pt]
   (-1.8,-1.5)  -- (-1.8,1.5) 
   node [black,midway,left=4pt,xshift=-2pt] {$\widehat{\mathbb{R}}_2$};
\draw [gray,decorate,decoration={brace,amplitude=5pt},yshift=-10pt]
   (1.5,-1.8)  -- (-1.5,-1.8) 
   node [black,midway,below=4pt,xshift=-2pt] {$\widehat{\mathbb{R}}_1$};
   
\fill[color=blue] (-.5,0) circle (1.1pt);
\fill[color=blue] (-.5,0.5) circle (1.1pt);
\fill[color=blue] (-.5,0.75) circle (1.1pt);
\fill[color=blue] (-.5,0.95) circle (1.1pt);
\fill[color=blue] (-.5,1.1) circle (1.1pt);
\fill[color=blue] (.5,0) circle (1.1pt);
\fill[color=blue] (.5,0.5) circle (1.1pt);
\fill[color=blue] (.5,0.75) circle (1.1pt);
\fill[color=blue] (.5,0.95) circle (1.1pt);
\fill[color=blue] (.5,1.1) circle (1.1pt);
\fill[color=blue] (-.5,-0.5) circle (1.1pt);
\fill[color=blue] (-.5,-0.75) circle (1.1pt);
\fill[color=blue] (-.5,-0.95) circle (1.1pt);
\fill[color=blue] (-.5,-1.1) circle (1.1pt);
\fill[color=blue] (.5,-0.5) circle (1.1pt);
\fill[color=blue] (.5,-0.75) circle (1.1pt);
\fill[color=blue] (.5,-0.95) circle (1.1pt);
\fill[color=blue] (.5,-1.1) circle (1.1pt);
\draw[line width=1pt, color=blue](-.5,0) -- (0.5,0);
\draw[line width=1pt, color=blue](-.5,0.5) -- (0.5,0.5);
\draw[line width=1pt, color=blue](-.5,-0.5) -- (0.5,-0.5);
\draw[line width=1pt, color=blue](-.5,0.75) -- (0.5,0.75);
\draw[line width=1pt, color=blue](-.5,-0.75) -- (0.5,-0.75);
\draw[line width=1pt, color=blue](-.5,0.95) -- (0.5,0.95);
\draw[line width=1pt, color=blue](-.5,-0.95) -- (0.5,-0.95);
\draw[line width=1pt, color=blue](-.5,1.1) -- (0.5,1.1);
\draw[line width=1pt, color=blue](-.5,-1.1) -- (0.5,-1.1);
\path (-0.15,1.35) node {$\vdots$};
\path (-0.15,-1.26) node {$\vdots$};

\end{tikzpicture}\end{center}
\caption{Vertical translates of $C:=[-1,1]\times\left\{0\right\}$}
\end{minipage}\end{figure}
\Par
\noindent Now consider the compact subset $C:=[-1,1]\times\left\{0\right\}$ of $\mathbb{R}_1\times\mathbb{R}_2$.  Then for any $n\in\mathbb{Z}$, $(0,n)\cdot C=[-1,1]\times\left\{n\right\}$.
\Par

Let $\mathcal{U}:=\left\{U_0,U_1,U_2,U_3\right\}$, where
\begin{center} $U_0:=(-\frac{3}{4},\frac{3}{4})\times\left((-\frac{1}{2},\infty)\times\left\{\beta_2\right\}\right)$, 
$U_1:=(-\frac{3}{4},\frac{3}{4})\times\left(\left\{\alpha_2\right\}\cup(-\infty,\frac{1}{2})\right)$
\\
$U_2:=\left(\left\{\alpha_1\right\}\cup(-\infty,-\frac{1}{2})\right)\times\widehat{\mathbb{R}}_2$, 
$U_3:=\left((\frac{1}{2},\infty)\cup\left\{\beta_1\right\}\right)\times\widehat{\mathbb{R}}_2$\end{center}

Then $\mathcal{U}$ is an open cover of $\widehat{\mathbb{R}^2}$, but $(0,n)\cdot C$ is not contained in any $U_i$ for any $n\in\mathbb{Z}$.

\begin{figure}[!h]
\begin{center}\begin{tikzpicture}[scale=1.15]
\fill[color=green,opacity=.4] (-1.5,-1.5) rectangle (-0.25,1.5);
\fill[color=blue,opacity=.4] (-0.375,-0.25) rectangle (0.375,1.5);
\fill[color=red,opacity=.4] (-0.375,0.25) rectangle (0.375,-1.5);
\fill[color=yellow,opacity=.4] (0.25,-1.5) rectangle (1.5,1.5);
\draw [dashed] (-0.25,1.5)--(-0.25,-1.5);
\draw [dashed] (-0.375,1.5) -- (-0.375,-.25) -- (0.375,-0.25) -- (0.375,1.5);
\draw [dashed] (-0.375,-1.5) -- (-0.375,.25) -- (0.375,0.25) -- (0.375,-1.5);
\draw [dashed] (0.25,1.5)--(0.25,-1.5);
\draw [-] (-1.5,1.5) -- (1.5,1.5) -- (1.5,-1.5) -- (-1.5,-1.5) --(-1.5,1.5);
\draw [-] (-1.5,0) -- (1.5,0);
\draw [-] (0,-1.5) -- (0,1.5);
\fill (-1.5,0) circle (1.1pt) node [left=1pt] {$(\alpha_1,0)$};
\fill (1.5,0) circle (1.1pt) node [right=1pt] {$(\beta_1,0)$};
\fill (0,-1.5) circle (1.1pt) node [below=1pt] {$(0,\alpha_2)$};
\fill (0,1.5) circle (1.1pt) node [above=1pt] {$(0,\beta_2)$};
\fill (-1.5,1.5) circle (1.1pt) node [above=1pt] {$(\alpha_1,\beta_2)$};
\fill (1.5,1.5) circle (1.1pt) node [above=1pt] {$(\beta_1,\beta_2)$};
\fill (-1.5,-1.5) circle (1.1pt) node [below=1pt] {$(\alpha_1,\alpha_2)$};
\fill (1.5,-1.5) circle (1.1pt) node [below=1pt] {$(\beta_1,\alpha_2)$};
\draw [gray,decorate,decoration={brace,amplitude=5pt},xshift=-25pt]
   (-1.8,-1.5)  -- (-1.8,1.5) 
   node [black,midway,left=4pt,xshift=-2pt] {$\widehat{\mathbb{R}}_2$};
\draw [gray,decorate,decoration={brace,amplitude=5pt},yshift=-10pt]
   (1.5,-1.8)  -- (-1.5,-1.8) 
   node [black,midway,below=4pt,xshift=-2pt] {$\widehat{\mathbb{R}}_1$};
   
\fill[color=blue] (-.5,0) circle (1.1pt);
\fill[color=blue] (-.5,0.5) circle (1.1pt);
\fill[color=blue] (-.5,0.75) circle (1.1pt);
\fill[color=blue] (-.5,0.95) circle (1.1pt);
\fill[color=blue] (-.5,1.1) circle (1.1pt);
\fill[color=blue] (.5,0) circle (1.1pt);
\fill[color=blue] (.5,0.5) circle (1.1pt);
\fill[color=blue] (.5,0.75) circle (1.1pt);
\fill[color=blue] (.5,0.95) circle (1.1pt);
\fill[color=blue] (.5,1.1) circle (1.1pt);
\fill[color=blue] (-.5,-0.5) circle (1.1pt);
\fill[color=blue] (-.5,-0.75) circle (1.1pt);
\fill[color=blue] (-.5,-0.95) circle (1.1pt);
\fill[color=blue] (-.5,-1.1) circle (1.1pt);
\fill[color=blue] (.5,-0.5) circle (1.1pt);
\fill[color=blue] (.5,-0.75) circle (1.1pt);
\fill[color=blue] (.5,-0.95) circle (1.1pt);
\fill[color=blue] (.5,-1.1) circle (1.1pt);
\draw[line width=1pt, color=blue](-.5,0) -- (0.5,0);
\draw[line width=1pt, color=blue](-.5,0.5) -- (0.5,0.5);
\draw[line width=1pt, color=blue](-.5,-0.5) -- (0.5,-0.5);
\draw[line width=1pt, color=blue](-.5,0.75) -- (0.5,0.75);
\draw[line width=1pt, color=blue](-.5,-0.75) -- (0.5,-0.75);
\draw[line width=1pt, color=blue](-.5,0.95) -- (0.5,0.95);
\draw[line width=1pt, color=blue](-.5,-0.95) -- (0.5,-0.95);
\draw[line width=1pt, color=blue](-.5,1.1) -- (0.5,1.1);
\draw[line width=1pt, color=blue](-.5,-1.1) -- (0.5,-1.1);
\path (-0.15,1.35) node {$\vdots$};
\path (-0.15,-1.26) node {$\vdots$};

\end{tikzpicture}\end{center}
\caption{An open cover of $\widehat{\mathbb{R}^2}$ whose elements contain no vertical translate of $C$}
\end{figure}
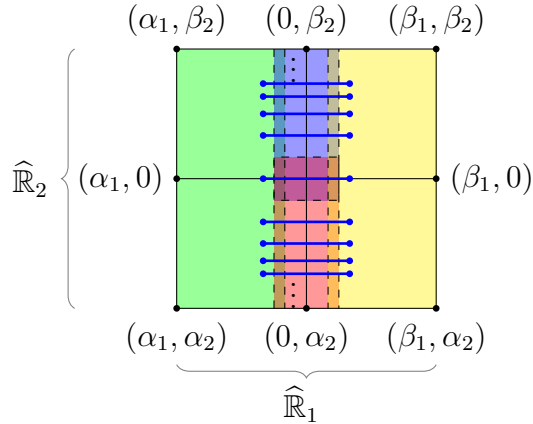

Therefore $\widehat{\mathbb{R}^2}$ does not satisfy the null condition with respect to the action of $\mathbb{Z}\times\mathbb{Z}$ on $\mathbb{R}\times\mathbb{R}$.\end{example}

\begin{example}  Now consider the $\mathcal{Z}$-compactification $\widehat{\mathbb{R}^2}'$ of $\mathbb{R}^2$ obtained instead by adjoining to $\mathbb{R}^2$ a circle $Z:=[0,2\pi]/\sim$, where $0\sim 2\pi$ and having as basis for its topology
\begin{equation*} \mathcal{B}:=\mathcal{B}_0\cup\mathcal{B}_{\partial}\end{equation*}
where $\mathcal{B}_0$ contains the standard open sets of $\mathbb{R}^2$, and $\mathcal{B}_{\partial}$ contains all sets of the form $B(z,R,\epsilon)$, where, for each $z\in Z,R>0,\epsilon>0$,
\begin{equation*}B(z,R,\epsilon):=\left\{(r,\theta)\in\mathbb{R}^2\ |\ r>R,\left|\theta-z\right|<\epsilon\right\}\cup\left\{z'\in Z\ |\ \left|z-z'\right|<\epsilon\right\}\end{equation*}
Note again, in Figure \ref{conenbhds} some examples of typical neighborhoods of boundary points in $\widehat{\mathbb{R}^2}'$.
\Par
\begin{figure}[!h]
\begin{center}\begin{tikzpicture}[scale=1.5]
\clip (-1.5,-1.5) rectangle (1.5,1.5);
\draw[dashed, fill=red!30] (100:3) arc (100:140:3) -- (140:1) arc (140:100:1) -- cycle;
\draw[dashed, fill=blue!30] (-60:3) arc (-60:40:3) -- (40:0.6) arc (40:-60:0.6) -- cycle;
\draw[dashed, fill=green!30] (190:3) arc (190:220:3) -- (220:0.1) arc (220:190:0.1) -- cycle;
\draw [-] (-1.5,1.5) -- (1.5,1.5) -- (1.5,-1.5) -- (-1.5,-1.5) --(-1.5,1.5);
\draw [-] (-1.5,0) -- (1.5,0);
\draw [-] (0,-1.5) -- (0,1.5);
\end{tikzpicture}\end{center}
\caption{Neighborhoods of boundary points in $\widehat{\mathbb{R}^2}$}
\label{conenbhds}\end{figure}
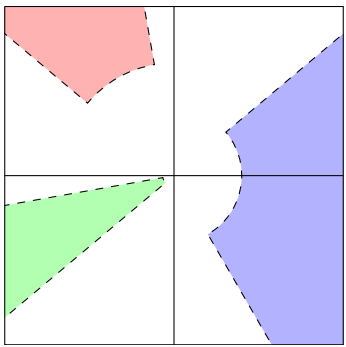

Now the variation in angles achieved by the translates of a given compactum in $\mathbb{R}^2$ shrinks as the compactum is pushed by the elements of $\mathbb{Z}\times \mathbb{Z}$ outside of metric balls of larger and larger radius.  Figure \ref{nullcondn} illustrates, for example, that all but finitely many translates $(0,n)\cdot C$ ($n\geq 0$) of the compactum $C=[-1,1]\times\left\{0\right\}$ fall into $B(\frac{\pi}{2},R,\epsilon)$, no matter how small $\epsilon$ is chosen.

\begin{figure}[!h]
\begin{center}$\frac{\pi}{2}$
\\
\begin{tikzpicture}[scale=1.5]
\clip (-1.5,-1.5) rectangle (1.5,1.5);
\draw[dashed, fill=green, fill opacity=.3] (45:3) arc (45:135:3) -- (135:0.25) arc (135:45:0.25) -- cycle;
\draw[dashed, fill=green, fill opacity=.3] (55:3) arc (55:125:3) -- (125:0.5) arc (125:55:0.5) -- cycle;
\draw[dashed, fill=green, fill opacity=.3] (65:3) arc (65:115:3) -- (115:0.75) arc (115:65:0.75) -- cycle;
\draw [-] (-1.5,1.5) -- (1.5,1.5) -- (1.5,-1.5) -- (-1.5,-1.5) --(-1.5,1.5);
\draw [-] (-1.5,0) -- (1.5,0);
\draw [-] (0,-1.5) -- (0,1.5);
\fill[color=blue] (-.5,0) circle (1.1pt);
\fill[color=blue] (-.5,0.5) circle (1.1pt);
\fill[color=blue] (-.5,0.75) circle (1.1pt);
\fill[color=blue] (-.5,0.95) circle (1.1pt);
\fill[color=blue] (-.5,1.1) circle (1.1pt);
\fill[color=blue] (.5,0) circle (1.1pt);
\fill[color=blue] (.5,0.5) circle (1.1pt);
\fill[color=blue] (.5,0.75) circle (1.1pt);
\fill[color=blue] (.5,0.95) circle (1.1pt);
\fill[color=blue] (.5,1.1) circle (1.1pt);
\fill[color=blue] (-.5,1.2) circle (1.1pt);
\fill[color=blue] (.5,1.2) circle (1.1pt);

\fill (0,1.5) circle (1.1pt) node [above=1pt] {$\frac{\pi}{2}$};

\draw[line width=1pt, color=blue](-.5,0) -- (0.5,0);
\draw[line width=1pt, color=blue](-.5,0.5) -- (0.5,0.5);
\draw[line width=1pt, color=blue](-.5,0.75) -- (0.5,0.75);
\draw[line width=1pt, color=blue](-.5,0.95) -- (0.5,0.95);
\draw[line width=1pt, color=blue](-.5,1.1) -- (0.5,1.1);
\draw[line width=1pt, color=blue](-.5,1.2) -- (0.5,1.2);

\end{tikzpicture}\end{center}
\caption{Translates of $C$ eventually fit into small neighborhoods of $\frac{\pi}{2}$}
\label{nullcondn}
\end{figure}
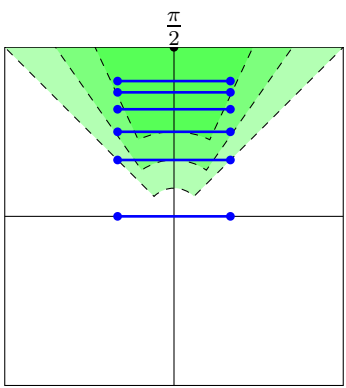
Therefore $\widehat{\mathbb{R}^2}'$ satisfies the null condition with respect to the action of $\mathbb{Z}\times\mathbb{Z}$ on $\mathbb{R}^2$. 
\Par
Perhaps the following depiction of this example is more appropriate in this paper, as its essence is analogous to the techniques used to prove Theorem \ref{zstructurethmdirprod}:
\Par
It is not difficult to see that the set $Z$ as described above is homeomorphic to $\left\{\alpha_1,\beta_1\right\}\ast\left\{\alpha_2,\beta_2\right\}$, where $\ast$ indicates a join.  In other words,
\begin{equation*}Z\approx\left\{\alpha_1,\beta_1\right\}\times\left\{\alpha_2,\beta_2\right\} \times[0,\infty]/\sim
\end{equation*}
where $\sim$ is the equivalence relation given by $(\alpha_1,\alpha_2,0)\sim(\alpha_1,\beta_2,0)$, $(\beta_1,\alpha_2,0)\sim(\beta_1,\beta_2,0)$, $(\alpha_1,\alpha_2,\infty)\sim(\beta_1, \alpha_2,\infty)$, and $(\alpha_1,\beta_2,\infty)\sim(\beta_1, \beta_2,\infty)$.
\Par
To each point of $\mathbb{R}^2$ we may assign a ``slope'' (say $(x,y)\longmapsto \left|\frac{y}{x}\right|$) and, in this example, points in $Z$ can be pulled in to $\mathbb{R}^2$ via a homotopy which keeps the slope coordinate constant.\end{example}

This is essentially the kind of structure we impose on $X\times Y$ when $(\widehat{X},\partial X)$ and $(\widehat{Y},\partial Y)$ are $\mathcal{Z}$-structures on $G$ and $H$, respectively, to prove Theorem \ref{zstructurethmdirprod}.
\Par
Suppose from this point forward that $(\widehat{X},\partial X)$ and $(\widehat{Y},\partial Y)$ are $\mathcal{Z}$-structures on $G$ and $H$, respectively.  We will denote by $\overline{x}$ (respectively, $\overline{y}$) a point in $\partial X$ (respectively, $\partial Y$), and by $\widehat{x}$ (respectively, $\widehat{y}$) a general point of $\widehat{X}$ (respectively, $\widehat{Y}$).
\Par
Since $\partial X$ and $\partial Y$ are $\mathcal{Z}$-sets in $\widehat{X}$ and $\widehat{Y}$, respectively, there exist homotopies $\alpha :\widehat{X}\times [0,1]\rightarrow \widehat{X}$ and $\beta:\widehat{Y}\times[0,1]\rightarrow \widehat{Y}$ such that $\alpha_0\equiv \text{id}_{\widehat{X}}$, $\beta_0\equiv\text{id}_{\widehat{Y}}$, $\alpha_t(\widehat{X})\subseteq X$ for all $t\in(0,1)$, and $\beta_t(\widehat{Y})\subseteq Y$ for all $t\in(0,1)$.  By Lemma \ref{sdrzcompactification}, we may assume in addition that $\alpha$ and $\beta$ are strong deformation retractions to base points $x_0\in X$ and $y_0\in Y$, so that $\alpha_1(\widehat{X})=\left\{x_0\right\}$, and $\beta_1(\widehat{Y})=\left\{y_0\right\}$.

\begin{definition}  A metric $d:X\times X\rightarrow [0,\infty)$ is \textbf{proper} if every closed metric ball in $X$ is compact.  A map $f:X\rightarrow Y$ is \textbf{proper} if for every compact $C\subseteq Y$, the preimage $f^{-1}(C)$ is compact in $X$.\end{definition}

\begin{lemma}\label{propermetric}  There exists a proper metric $d$ on $X$.
\end{lemma}
\begin{proof}  First, note that, since $\widehat{X}$ is metrizable (it is an ER), we may choose a metric $\widehat{d}$ on $\widehat{X}$.  Then $\widehat{D}:(\widehat{X}\times [0,\infty))\times(\widehat{X}\times [0,\infty))\rightarrow[0,\infty)$ defined by 
\begin{equation*}\widehat{D}((x_1,t_1),(x_2,t_2)):=\sqrt{(\widehat{d}(x_1,x_2))^2+\left|t_1-t_2\right|^2}\end{equation*}
is a proper metric on $\widehat{X}\times [0,\infty)$.
\Par
Let $f:\widehat{X}\rightarrow [0,1]$ be a continuous function satisfying $f(x_0)=0$, $f(\partial X)=\left\{1\right\}$, and $f(x)\in(0,1)$ if $x\in X\backslash\left\{x_0\right\}$.
\Par
Let $h:[0,\infty)\rightarrow[0,1)$ be a homeomorphism, and consider the graph
\\
$G:=\left\{(x,f(h(x)))\ |\ x\in X\right\}\subseteq \widehat{X}\times[0,\infty)$ of $f\circ h$.  Since $\widehat{X}\times[0,\infty)$ is a proper metric space and $g:X\rightarrow \widehat{X}\times[0,\infty)$ with $g(x)=(x,f(h(x)))$ is a proper embedding of $X$ in $\widehat{X}\times[0,\infty)$, then $X$ inherits a proper metric $d$ from $\widehat{X}\times[0,\infty)$.
\end{proof}

From now on, we will assume that $(X,\rho)$ and $(Y,\tau)$ are proper metric spaces, and that $\overline{\rho}$ and $\overline{\tau}$ are metrics on $\widehat{X}$ and $\widehat{Y}$, respectively.

\begin{lemma}\label{defpq}  There exists a proper map $p:X\rightarrow [0,\infty)$ having the following properties: 
\\
\tab (i)  The variation of $p$ over translates of a given compactum in $X$ is bounded, i.e.
\begin{center} $R_p(C):=\sup\left\{\max\left\{p(x)-p(x')\ |\ x,x'\in gC\right\}\ |\ g\in G\right\}<\infty$\hspace{.25in}$(\dagger)$\end{center}
\tab for any compactum $C$ in $X$,
\\
and
\\
\tab (ii)  For some sequence $1=t_0>t_1>t_2>\cdots >0$, we have
\begin{center}$p\left(\alpha(\partial X\times [t_i,t_{i-1}))\right)\subseteq (i-1,i+1]$\hspace{.25in}$(\dagger\dagger)$\end{center} 
\end{lemma}

\begin{proof}  Let $t_0:=1$.  Let $C_1$ be a connected compact subset of $X$ containing $x_0$ with the property that the translates of $C_1$ cover $X$, i.e. $\bigcup\limits_{g\in G}gC_1=X$.  Since $C_1$ is compact, there exists $r_1>0$ such that $B_{\rho}(x_0, r_1)\supseteq C_1$.
\Par
Let $t_1\in(0,1)$ be such that $\alpha(\partial X\times [0,t_1))\cap\overline{B_{\rho}(x_0,r_1)}=\emptyset$, and choose $r_2'$ so that $B_{\rho}(x_0, r_2')\supseteq\alpha(\partial X\times[t_1,1])$.
\Par
Choose $r_2$ such that
\begin{center}$B_{\rho}(x_0, r_2)\supseteq\overline{B_{\rho}(x_0,r_2')}\cup\Bigl(\cup\left\{gC_1\ |\ gC_1\cap B_{\rho}(x_0,r_1)\neq\emptyset\right\}\Bigr)$\end{center}
and $t_2\in(0,1)$ such that $\alpha(\partial X\times[0,t_2))\cap\overline{B_{\rho}(x_0,r_2)}=\emptyset$.
\Par
Continue inductively.
\Par
For each i, let $r_i'>0$ satisfy $B_{\rho}(x_0,r_i')\supseteq\alpha(\partial X\times[t_{i-1},1])$.  Then choose $r_i>0$ so that
\begin{eqnarray*}B_{\rho}(x_0,r_i)\supseteq\overline{B_{\rho}(x_0,r_i')}\cup\bigl(\cup\left\{gC_1\ |\ gC_1\cap B_{\rho}(x_0,r_{i-1})\neq\emptyset\right\}\bigr)\text{ \hspace{.1in} } (\star)\end{eqnarray*}
and $t_i\in(0,1)$ such that
\begin{center}$\alpha(\partial X\times [0,t_i))\cap\overline{B_{\rho}(x_0,r_i)}=\emptyset$. \hspace{1in} ($\star\star$)\end{center}
We have $0<r_1<r_2<\cdots$ with $r_i\rightarrow\infty$ as $i\rightarrow\infty$, so that $X=\bigcup\limits_{i=1}^{\infty}B_{\rho}(x_0,r_i)$.
\Par
Moreover, we have $1=t_0>t_1>t_2>\cdots>0$ with $t_i\rightarrow 0$ as $i\rightarrow\infty$.
\Par
Define $p:X\rightarrow [0,\infty)$ to be a piecewise rescaling of the map $\rho(\cdot,x_0)$ measuring distance to the point $x_0$ in such a way that $p(x_0)=0$, and $p(\overline{B_{\rho}(x_0,r_i)}-B_{\rho}(x_0,r_{i-1}))=[i-1,i]$.  Since $(X,\rho)$ is a proper metric space, it is clear that $p$ is a proper map.

\begin{claim}  \label{pworks}The map $p:X\rightarrow[0,\infty)$ satisfies $(\dagger)$.\end{claim}

\textit{Proof.}  First we note that $(\dagger)$ holds for $C_1$:
\Par
For each $g\in G$, let $i_g:=\min\left\{k\in\mathbb{N}\ |\ gC_1\cap B_{\rho}(x_0,r_k)\neq\emptyset\right\}$.  Then, by definition, $p(gC_1)\subseteq [i_g-1,i_g+1]$.  Thus
\begin{center}$\max\left\{p(x)-p(x')\ |\ x,x'\in gC_1\right\}\leq 2$ for every $g\in G$,\end{center}
so $R_p(C_1)\leq 2$.
\Par
Now consider any compactum $C$ in $X$.  We may assume, without loss of generality, that $x_0\in C$.  Since $C$ is compact, it is contained in a metric ball in $X$, so there is a minimal finite collection $\left\{g_1,g_2,\ldots,g_{k_C}\right\}\subseteq G$ so that $\left\{g_1C_1,g_2C_1,\ldots,g_{k_C}C_1\right\}$ covers $C$ and $\bigcup_{i=1}^{k_C}g_iC_1$ is connected.  Then, in fact, given any $g\in G$, $\displaystyle\bigcup_{i=1}^{k_C}gg_iC_1$ is connected and contains $gC$.  Therefore, by connectedness and a simple inductive argument, 
\begin{center}$\max\left\{p(x)-p(x')\ |\ x,x'\in gC\right\}\leq 2k_C$ for all $g\in G$.\end{center}

\noindent Hence $R_p(C)\leq 2k_C<\infty$ for any compactum $C$ in $X$.
\Par
\noindent By Claim \ref{pworks}, we have constructed a proper map $p:X\rightarrow [0,\infty)$ satisfying $(\dagger)$ for any compactum $C$ in $X$.
\Par
Moreover, ($\star$) and ($\star\star$) guarantee that ($\dagger\dagger$) is satisfied by the constructed $p$.
\end{proof}

\noindent Certainly we may define, using the same methods, a proper map $q:Y\rightarrow [0,\infty)$ satisfying conditions analogous to ($\dagger$) and ($\dagger\dagger$).

\begin{lemma}  \label{alphahatbetahat} There are reparametrizations $\widehat{\alpha}$ and $\widehat{\beta}$ of the homotopies $\alpha$ and $\beta$ so that $p(\widehat{\alpha}(\overline{x},t))\in\left[\frac{1}{t}-1,\frac{1}{t}+2\right]$ and $q(\widehat{\beta}(\overline{y},t))\in\left[\frac{1}{t}-1,\frac{1}{t}+2\right]$ for all $t\in(0,1]$, $\overline{x}\in\partial{X}$, $\overline{y}\in\partial{Y}$.\end{lemma}

\begin{proof}  Note that, using the notation from Lemma \ref{defpq}, we have, for any $t\in(0,1)$, some $i\in\mathbb{N}$ so that $t_i\leq t <t_{i-1}$.  Condition ($\dagger\dagger$) gives, then, that $p(\alpha(\partial X\times\left\{t\right\}))\subseteq(i-1,i+1]$.
\Par
Let $\xi:[0,1]\rightarrow[0,1]$ be the piecewise linear homeomorphism satisfying $\xi(0)=0$, $\xi(1)=1$, and $\xi(\frac{1}{i})=t_{i-1}$ for all $i\in\mathbb{N}$, and define $\widehat{\alpha}:\widehat{X}\times[0,1]\rightarrow\widehat{X}$ by $\widehat{\alpha}(\widehat{x},t):=\alpha(\widehat{x},\xi(t))$ for all $\widehat{x}\in\widehat{X}$ and all $t\in[0,1]$.
\Par
Now we have arranged that, given $t\in[0,1]$ and $i\in\mathbb{N}$ such that $t\in[\frac{1}{i},\frac{1}{i-1})$, $\xi(t)\in[t_i,t_{i-1})$, so
\begin{center} $p(\widehat{\alpha}(\overline{x},t))=p(\alpha(\overline{x},\xi(t)))\in(i-1,i+1]\subseteq(\frac{1}{t}-1,\frac{1}{t}+2]$\end{center}
for any $\overline{x}\in\partial X$.
\Par
Moreover, $p(\widehat{\alpha}(\overline{x},1))=p(\alpha(\overline{x},1))=p(x_0)=0\in[0,3]$ for any $\overline{x}\in\partial X$, so $\widehat{\alpha}$ satisfies the requirement at $t=1$.
\Par
Define $\widehat{\beta}$ similarly, and the result holds.
\end{proof}

\begin{definition} We define $\widehat{X\times Y}$ as follows:
\Par
The \textbf{join} $\partial X\ast\partial Y$ of the boundaries $\partial X$ and $\partial Y$ is:
\\
\tab $\partial X\ast\partial Y=\partial X\times \partial Y\times [0,\infty]/\sim$,
\\
where $\sim$ is the equivalence relation generated by $\left\langle\overline{x},\overline{y},\mu\right\rangle\sim\left\langle\overline{x}',\overline{y}',\mu'\right\rangle$ if and only if ($\mu=\mu'=0$ and $\overline{x}=\overline{x}'$) or ($\mu=\mu'=\infty$ and $\overline{y}=\overline{y}')$.
\Par
We will denote by $\left\langle\overline{x},0\right\rangle$ the equivalence class under $\sim$ containing $\left\langle\overline{x},\overline{y},0\right\rangle$ for all $\overline{y}\in\partial Y$, and by $\left\langle\overline{y},\infty\right\rangle$ the equivalence class containing $\left\langle\overline{x},\overline{y},\infty\right\rangle$ for all $\overline{x}\in\partial X$.
\Par
Now we define a slope function $\mu:X\times Y\rightarrow [0,\infty]$ by \begin{displaymath}
   \mu(x,y) = \left\{
     \begin{array}{lr}
       \displaystyle\frac{q(y)}{p(x)} & \text{if }p(x)\neq 0\\
       \infty & \text{if }p(x)=0
     \end{array}
   \right.
\end{displaymath}
\Par
As a set, $\widehat{X\times Y}:=(X\times Y)\cup(\partial X\ast\partial Y)$.
\Par
The topology on $\widehat{X\times Y}$ is generated by the basis $\mathcal{B}:=\mathcal{B}_0\cup\mathcal{B}_{\partial}$, where
\begin{center} $\mathcal{B}_0:=\left\{U\times V\ |\ U\text{ is open in }X, V\text{ is open in }Y\right\}$
\\
$\mathcal{B}_{\partial}:=\left\{U(z,\epsilon)\ |\ z\in\partial X\ast\partial Y,\epsilon>0\right\}$
\end{center}
\Par
The neighborhoods $U(z,\epsilon)$ of boundary points are defined by:
\Par
Given $\overline{x}\in\partial X$ and $\epsilon >0$,

\begin{center} $U(\left\langle\overline{x},0\right\rangle,\epsilon):=\left\{(x,y)\ |\ \overline{\rho}(x,\overline{x}),\mu(x,y)<\epsilon\right\}$
\\
$\cup\left\{\left\langle\overline{x}',\overline{y}',\mu'\right\rangle\ |\ \overline{\rho}(\overline{x},\overline{x}'),\mu'<\epsilon\right\}$
\\
$\cup\left\{\left\langle\overline{x}',0\right\rangle\ |\ \overline{\rho}(\overline{x},\overline{x}')<\epsilon\right\}$
\end{center}

\noindent For $\overline{y}\in\partial Y$ and $\epsilon>0$,

\begin{center}$U(\left\langle\overline{y},\infty\right\rangle,\epsilon):=\{(x,y)\ |\ \overline{\tau}(y,\overline{y}),\frac{1}{\mu(x,y)}<\epsilon\}$
\\
$\cup\{\left\langle\overline{x}',\overline{y}',\mu'\right\rangle\ |\ \overline{\tau}(\overline{y},\overline{y}'),\frac{1}{\mu'}<\epsilon\}$
\\
$\cup\left\{\left\langle\overline{y}',\infty\right\rangle\ |\ \overline{\tau}(\overline{y},\overline{y}')<\epsilon\right\}$\end{center}

\noindent For $\left\langle\overline{x},\overline{y},\mu\right\rangle\in\partial X\times\partial Y\times (0,\infty)$ and $\epsilon<\mu$,

\begin{center}$U(\left\langle\overline{x},\overline{y},\mu\right\rangle,\epsilon):=\left\{(x,y)\ |\ \overline{\rho}(x,\overline{x}),\overline{\tau}(y,\overline{y}),\left|\mu(x,y)-\mu\right|<\epsilon\right\}$
\\
$\cup\left\{\left\langle\overline{x}',\overline{y}',\mu'\right\rangle\ |\ \overline{\rho}(\overline{x},\overline{x}'),\overline{\tau}(\overline{y},\overline{y}'),\left|\mu-\mu'\right|<\epsilon\right\}$\end{center}
\textit{Note:}  Recall that $\overline{\rho}$ and $\overline{\tau}$ are metrics on the compactifications $\widehat{X}$ and $\widehat{Y}$, respectively.\end{definition}

\begin{proposition}  $\widehat{X\times Y}$ is a compactification of $X\times Y$.\end{proposition}
\begin{proof}  We first observe that the topology inherited by $X\times Y$ as a subspace of $\widehat{X\times Y}$ is the same as the original topology on $X\times Y$.
\Par
It remains to show that $\widehat{X\times Y}$ is compact.
\Par
Let $\mathcal{U}$ be an open cover of $\widehat{X\times Y}$ by basic open sets.  Since $\partial X\ast \partial Y$ is compact, we may choose a finite subset $\mathcal{U}_{\partial}=\left\{U_i\right\}_{i=1}^k$ of $\mathcal{U}$ which covers $\partial X\ast \partial Y$.
\\
\begin{claim}  \label{delta} There exists $1>\delta>0$ such that for every $z\in\partial X\ast \partial Y$ there is some $i\in\left\{1,\ldots,k\right\}$ such that $U(z,\delta)\subseteq U_i$.\end{claim}

\noindent\textit{Proof.}  For each $i=1,\ldots, k$, define $\eta_i:\partial X\ast\partial Y\rightarrow[0,\infty)$ by
\begin{displaymath}
   \eta_i(z) := \left\{
     \begin{array}{lr}
       0 & \text{if }z\notin U_i\\
       \sup\left\{\eta>0\ |\ U\left(z,\eta\right)\subseteq U_i\right\}& \text{if }z\in U_i
     \end{array}
   \right.
\end{displaymath}
\\
Then each $\eta_i$ is a continuous function, and for each $z\in \partial X\ast\partial Y$, there is some $i\in\left\{1,\ldots, k \right\}$ so that $\eta_i(z)>0$.
\Par
Now $\eta:=\max\left\{\eta_i\ |\ i=1,\ldots k\right\}$ is a continuous and strictly positive function on the compact set $\partial X\ast\partial Y$, so
\begin{equation*}\delta:=\min\left\{\delta',\frac{1}{2}\right\}\end{equation*}
where
\begin{equation*}\delta':=\min\left\{\eta(z)\ |\ z\in \partial X\ast\partial Y\right\}\end{equation*}

\noindent is a positive number which satisfies the desired condition.
\Par
We will show that there is a compactum $C\subseteq X\times Y$ such that if $(x,y)\notin C$, then $(x,y)\in U(z,\delta)$ for some $z\in\partial X\ast\partial Y$.
\\

\begin{claim}  \label{claimX}Given a compactum $J\subseteq X$, there is a compactum $P_J\subseteq Y$ such that if $(x,y)\in J\times (Y\backslash P_J)$ then $(x,y)\in U(\left\langle\overline{y},\infty\right\rangle,\delta)$ for some $\overline{y}\in\partial Y$.\end{claim}

\noindent\textit{Proof.}  Let $M_J:=\max\left\{p(x)\ | \ x\in J\right\}$, and choose $P_J$ sufficiently large so that if $y\notin P_J$, then $\overline{\tau}(y,\partial Y)<\delta$ and $q(y)>M_J\cdot\frac{1}{\delta}$.
\Par
Then if $(x,y)\in J\times (Y\backslash P_J)$, there is some $\overline{y}\in\partial Y$ such that $\overline{\tau}(y,\overline{y})<\delta$, and $\mu(x,y)=\frac{q(y)}{p(x)}>\frac{1}{\delta}$, so $(x,y)\in U(\left\langle\overline{y},\infty\right\rangle,\delta)$.
\Par
Similarly, we have:

\begin{claim}  \label{claimY}Given a compactum $K\subseteq Y$, there is a compactum $Q_K\subseteq X$ such that if $(x,y)\in (X\backslash Q_K)\times K$, then $(x,y)\in U(\left\langle\overline{x},0\right\rangle,\delta)$ for some $\overline{x}\in\partial X$.\end{claim}
\vspace{.2in}

\noindent We define
\begin{equation*}C:=(C_X\times P_{C_X})\cup(Q_{C_Y}\times C_Y),\end{equation*}

\noindent where
\begin{equation*}C_X:=\widehat{X}\backslash B_{\overline{\rho}}(\partial X,\delta)\text{ and }
C_Y:=\widehat{Y}\backslash B_{\overline{\tau}}(\partial Y,\delta).\end{equation*}

\noindent Now suppose $(x,y)\in (X\times Y)\backslash C$.  If $x\in C_X$ or $y\in C_Y$, then Claim \ref{claimX} or \ref{claimY} gives the result.  Otherwise we have $x\notin C_X$ and $y\notin C_Y$, which implies that there are $\overline{x}\in\partial X$ and $\overline{y}\in\partial Y$ such that $\overline{\rho}(x,\overline{x}),\overline{\tau}(y,\overline{y})<\delta$.  Therefore $(x,y)\in U(\left\langle\overline{x},\overline{y},\mu(x,y)\right\rangle,\delta)$, and the proposition is proved.
\end{proof}

Let us clarify here a future abuse of notation: by $\partial X\subseteq\partial X\ast \partial Y$ (respectively $\partial Y\subseteq\partial X\ast \partial Y$), we mean the homeomorphic copy $\partial X\times\partial Y\times \left\{0\right\}/\sim$ (respectively $\partial X\times\partial Y\times \left\{\infty\right\}/\sim$) of $\partial X$ (respectively $\partial Y$) in $\partial X\ast\partial Y$.
\begin{proposition}\label{null}$\widehat{X\times Y}$ satisfies the null condition with respect to the action of $G\times H$ on $X\times Y$.\end{proposition}

\begin{proof}  Consider a compactum $C\times D$ in $X\times Y$, where $C$ is compact in $X$ and $D$ is compact in $Y$.  Let $\mathcal{U}$ be an open cover of $\widehat{X\times Y}$ by basic open sets.  We may assume, without loss of generality, that $\mathcal{U}$ is finite, since $\widehat{X\times Y}$ is compact.
\Par
Let $\mathcal{U}_{\partial}=\left\{U_i\right\}_{i=1}^k$ denote the finite subset of $\mathcal{U}$ which covers $\partial X\ast \partial Y$.
\Par
Choose $1>\delta>0$ as in Claim \ref{delta}.
\Par
\noindent\textit{Notation:}  Let $\text{diam}_{\mu}(A):=\sup\left\{\mu(x,y)-\mu(x',y')\ |\ (x,y),(x',y')\in A\right\}$ for any $A\subseteq X\times Y$.
\Par
We also denote by $W$ the set of points $\left\langle \overline{x},\overline{y},\mu\right\rangle$ in $\partial X\ast\partial Y$ with $0<\mu<\infty$.
\\ 

\begin{claim}  \label{close}If $\left(gC\times hD\right)\cap U\left(w,\frac{\delta}{2}\right)\neq\emptyset$ for some $w\in W$ and
\\ $\text{diam}_{\overline{\rho}}gC,\text{diam}_{\overline{\tau}}hD,\text{diam}_{\mu}\left(gC\times hD\right)<\frac{\delta}{2}$, then there is some
\\ $i\in\left\{1,\ldots, k\right\}$ so that $gC\times hD\subseteq U_i$.\end{claim}

\textit{Proof.}  This claim follows easily from simple calculations using the triangle inequality and the definition of $U\left(w,\epsilon\right)$.
\Par

Define
\begin{eqnarray*}R_p:=\sup\left\{\max\left\{p(x)-p(x')\ |\ x,x'\in gC\right\}\ |\ g\in G\right\}
\\
R_q:=\sup\left\{\max\left\{q(y)-q(y')\ |\ y,y'\in hD\right\}\ |\ h\in H\right\}\end{eqnarray*}
\\
Note that by Lemma \ref{defpq}, both $R_p$ and $R_q$ are finite.
\\
\begin{claim}  \label{PJ} Given a compactum $J\subseteq X$, there is a compactum $P_J\subseteq Y$ so that if $(gC\times hD)\cap(J\times P_J)=\emptyset$, but $gC\cap J\neq\emptyset$, then there is some $\overline{y}\in\partial Y$ such that $gC\times hD\subseteq U(\left\langle\overline{y},\infty\right\rangle,\delta)$.\end{claim}

\textit{Proof.}  Since the action of $G$ on $X$ is proper, then $\left|\left\{g\in G\ |\ gC\cap J\neq\emptyset\right\}\right|<\infty$, so $M_J:=\max\left\{p(x)\ |\ x\in gC,\ gC\cap J\neq\emptyset\right\}<\infty$ since $C$ is compact.
\Par
Choose $P_J$ sufficiently large so that if $hD\not\subseteq P_J$, then
\begin{eqnarray} q(y)>M_J\cdot \frac{1}{\delta} \ \ \forall y\in hD\label{i}
\\
\overline{\tau}(hD,\partial Y)<\frac{\delta}{2}\label{ii}
\\
\text{diam}_{\overline{\tau}}hD<\frac{\delta}{2}\label{iii}\end{eqnarray}
\\
Note that (\ref{i}) can be achieved by the properness of the function $q$, (\ref{ii}) by cocompactness of the action of $H$ on $Y$, and (\ref{iii}) by the fact that $\widehat{Y}$ satisfies the null condition with respect to the action of $H$ on $Y$.
\Par
Now if $(gC\times hD)\cap(J\times P_J)=\emptyset$ and $gC\cap J\neq \emptyset$, then $hD\not\subseteq P_J$, so by (\ref{ii}) and (\ref{iii}), we have $hD\subseteq B_{\overline{\tau}}(\overline{y},\delta)$ for some $\overline{y}\in\partial Y$.
\Par
Therefore for any $(x,y)\in gC\times hD$, we have \begin{center}$\overline{\tau}(y,\overline{y})<\delta$\end{center}
and 
\begin{center}$\displaystyle\mu(x,y)=\frac{q(y)}{p(x)}>\frac{M_J\cdot \frac{1}{\delta}}{M_J}=\frac{1}{\delta}$.\end{center}
Hence, $gC\times hD\subseteq U\left(\left\langle\overline{y},\infty\right\rangle,\delta\right)$, and the claim is proved.
\Par
Clearly we may use analogous techniques to obtain:

\begin{claim}  \label{QK}Given a compactum $K\subseteq Y$, there is a compactum $Q_K\subseteq X$ so that if $(gC\times hD)\cap(K\times Q_K)=\emptyset$, but $hD\cap K\neq\emptyset$, then there is some $\overline{x}\in\partial X$ such that $gC\times hD\subseteq U(\left\langle\overline{x},0\right\rangle,\delta)$.\end{claim}

\noindent Now choose a compact subset $J\subseteq X$ containing $x_0$ such that if $gC\not\subseteq J$, then
\begin{eqnarray*} \text{diam}_{\overline{\rho}} gC<\frac{\delta}{4}
\\
p(x)>\frac{4}{\delta}\left(R_q+\frac{1}{\delta}\cdot R_p\right)\ \forall x\in gC
\\
\overline{\rho}(gC,\partial X)<\frac{\delta}{4}\end{eqnarray*}

and a compact subset $K\subseteq Y$ containing $y_0$ such that if $hD\not\subseteq K$, then

\begin{eqnarray*} \text{diam}_{\overline{\tau}} hD<\frac{\delta}{4}
\\
\overline{\tau}(hD, \partial Y)<\frac{\delta}{4}\end{eqnarray*}
\\
Let $P_J\subseteq Y$ and $Q_K\subseteq X$ be as in Claims \ref{PJ} and \ref{QK}, respectively.

\begin{claim}  \label{getbiggamma}If $(gC\times hD)\cap\left[\left(J\times P_J\right)\cup\left(Q_K\times K\right)\right]=\emptyset$, then $gC\times hD$ is contained in a single element of $\mathcal{U}$.\end{claim}

\textit{Proof.}  By Claims \ref{PJ}, \ref{QK}, and the choice of $\delta$, if $gC\cap J\neq\emptyset$ or $hD\cap K\neq\emptyset$, then we are done.
\Par
Assume that $gC\cap J=hD\cap K=\emptyset$.  Then $\text{diam}_{\overline{\rho}} gC,\text{diam}_{\overline{\tau}}hD<\frac{\delta}{4}$, and there exist $\left(\overline{x},\overline{y}\right)\in\partial X\times\partial Y$ and $\left(\widehat{x},\widehat{y}\right)\in gC\times hD$ such that $\overline{\rho}(\widehat{x},\overline{x}),\overline{\tau}(\widehat{y},\overline{y})<\frac{\delta}{4}$.
\Par
\textit{Case 1:}  There exists $\left(x',y'\right)\in gC\times hD$ such that $\delta\leq\mu(x',y')\leq \frac{1}{\delta}$.
\Par
Then $\left(\overline{x},\overline{y},\mu(x',y')\right)\in W$, and since
\begin{eqnarray*}\overline{\rho}(x',\overline{x})\leq \overline{\rho}(x',\widehat{x})+\overline{\rho}(\widehat{x},\overline{x})<\frac{\delta}{4}+\frac{\delta}{4}=\frac{\delta}{2}
\\
\overline{\tau}(y',\overline{y})\leq \overline{\tau}(y',\widehat{y})+\overline{\tau}(\widehat{y},\overline{y})<\frac{\delta}{4}+\frac{\delta}{4}=\frac{\delta}{2}\end{eqnarray*}
we have $(x',y')\in (gC\times hD)\cap U\left(\left(\overline{x},\overline{y},\mu(x',y')\right),\frac{\delta}{2}\right)$.
\Par
Moreover, for any $\left(x,y\right)\in gC\times hD$, we have
\begin{align*}\left|\mu(x,y)-\mu(x',y')\right|&=\left|\mu(x,y)-\mu(x,y')+\mu(x,y')-\mu(x',y')\right|
\\
&=\left|\frac{q(y)}{p(x)}-\frac{q(y')}{p(x)}+\frac{q(y')}{p(x)}-\frac{q(y')}{p(x')}\right|
\\
&\leq\frac{1}{p(x)}\cdot\left|q(y)-q(y')\right|+\frac{q(y')}{p(x')}\cdot\frac{\left|p(x')-p(x)\right|}{p(x)}
\\
&=\frac{1}{p(x)}\cdot\left|q(y)-q(y')\right|+\mu(x',y')\cdot\frac{\left|p(x')-p(x)\right|}{p(x)}
\\
&<\frac{1}{\frac{4}{\delta}\cdot(R_q+\frac{1}{\delta}\cdot R_p)}\cdot R_q+\frac{1}{\delta}\cdot\frac{R_p}{\frac{4}{\delta}\cdot(R_q+\frac{1}{\delta}\cdot R_p)}
\\
&=\frac{\delta}{4}
\end{align*}
\\
Hence $\text{diam}_{\mu}gC\times hD<\frac{\delta}{2}$, so the conditions of Claim \ref{close} are satisfied, and $gC\times hD$ is contained in a single element of $\mathcal{U}$.
\Par
\textit{Case 2:}  There is no $\left(x',y'\right)\in gC\times hD$ with $\delta\leq \mu(x',y)\leq \frac{1}{\delta}$.
\Par
Then we have $\mu(x,y)<\delta$ for all $(x,y)\in gC\times hD$, or $\mu(x,y)>\frac{1}{\delta}$ for all $(x,y)\in gC\times hD$.
\Par
In the case where $\mu(x,y)<\delta$ for all $(x,y)\in gC\times hD$, we have
\begin{center}$\overline{\rho}(x,\overline{x})\leq \overline{\rho}(x,\widehat{x})+\overline{\rho}(\widehat{x},\overline{x})<\frac{\delta}{4}+\frac{\delta}{4}=\frac{\delta}{2}$\end{center}
for all $(x,y)\in gC\times hD$, so that, in fact, $gC\times hD\subseteq U(\overline{x},\delta)$.
\Par
A similar argument shows that if $\mu(x,y)>\frac{1}{\delta}$ for all $(x,y)\in gC\times hD$, then $gC\times hD\subseteq U((\overline{y},\delta)$.
\Par
This proves the claim.
\Par
\noindent Finally, let
\begin{center}$\Gamma:=\left\{(g,h)\in G\times H\ |\ \left(gC\times hD\right)\cap
\left[\left(J\times P_J\right)\cup\left(Q_K\times K\right)\right]\neq \emptyset\right\}$.\end{center}
Then $\Gamma$ is finite by cocompactness of the actions of $G$ and $H$ on $X$ and $Y$, respectively, and Claim \ref{getbiggamma} shows that if $(g,h)\notin\Gamma$, then $gC\times hD$ is contained in a single element of the original cover $\mathcal{U}$.
\Par
Therefore $\widehat{X\times Y}$ satisfies the null condition with respect to the action of $G\times H$ on $X\times Y$.
\end{proof}

To prove that $\widehat{X\times Y}$ is an ANR, we will construct a homotopy $\gamma:\widehat{X\times Y}\times[0,1]\rightarrow \widehat{X\times Y}$ which pulls $\widehat{X\times Y}$ off of $\partial X\ast \partial Y$ into the ANR $X\times Y$.  In analogy with the CAT(0) case, we describe the homotopy by first constructing a ``ray`` from the base point $(x_0,y_0)$ to each point of $\widehat{X\times Y}$.  The homotopy $\gamma$ then pulls points inward along these rays.  The subtle point of the argument, and the key to obtaining continuity, is the parametrization of the rays in such a way that the slope function $\mu$ is respected near $\partial X\ast \partial Y$.  After $\gamma$ is constructed, we apply Theorem \ref{hanner} to conclude that $\widehat{X\times Y}$ is an ANR.  The existence of $\gamma$ will also imply that $\partial X\ast \partial Y$ is a $\mathcal{Z}$-set in $\widehat{X\times Y}$:
\Par
Define $\alpha':\widehat{X}\times[0,\infty)\rightarrow X$ and $\beta':\widehat{Y}\times[0,\infty)\rightarrow Y$ by \begin{center}$\alpha'(\widehat{x},t):=\widehat{\alpha}(\widehat{x},\delta(t))$ and $\beta'(\widehat{y},t):=\widehat{\beta}(\widehat{y},\delta(t))$\end{center}
for all $\widehat{x}\in\widehat{X},\widehat{y}\in\widehat{Y},t\in[0,\infty)$, where $\delta:[0,\infty)\rightarrow (0,1]$ is given by $\delta(t):=\frac{1}{1+t}$ and $\widehat{\beta}$ and $\widehat{\alpha}$ are as defined in Lemma \ref{alphahatbetahat}.
\Par
Now a simple calculation shows that for any $t\in[0,\infty),\overline{x}\in\partial{X},\overline{y}\in\partial{Y}$, we have $p(\alpha'(\overline{x},t)),q(\beta'(\overline{y},t))\in(t-1,t+3)$.  This will allow us to construct rays in $X\times Y$ which respect the slope function $\mu$ by controlling the speeds at which $\alpha'$ and $\beta'$ are traced.
\Par
Let $\gamma':\widehat{X\times Y}\times[0,\infty)\rightarrow X\times Y$ be given by:
\\
 $\bullet\ \ \gamma'((x,y),t):=\left(\alpha'\left(x,\frac{t}{\sqrt{\left(\mu(x,y)\right)^2+1}}\right),\beta'\left(y,\frac{\mu(x,y)\cdot t}{\sqrt{\left(\mu(x,y)\right)^2+1}}\right)\right)$ if $(x,y)\in X\times Y,\ t\geq 0$
\Par
$\bullet\ \ \gamma'(\left\langle\overline{x},\overline{y},\mu\right\rangle,t):=\left(\alpha'\left(\overline{x},\frac{t}{\sqrt{\mu^2+1}}\right),\beta'\left(\overline{y},\frac{\mu\cdot t}{\sqrt{\mu^2+1}}\right)\right)$ if $\left\langle\overline{x},\overline{y},\mu\right\rangle\in \partial X\ast \partial Y,\ 0<\mu<\infty,\ \ t\geq 0$
\Par
$\bullet\ \ \gamma'\left(\left(\overline{x},0\right),t\right):=\left(\alpha'\left(\overline{x},t\right),y_0\right)$ if $\overline{x}\in\partial X,\ t\geq 0$
\Par
$\bullet \ \ \gamma'\left(\left(\overline{y},\infty\right),t\right):=\left(x_0,\beta'\left(\overline{y},t\right)\right)$ if $\overline{y}\in\partial Y,\ t\geq 0$
\Par
The map $\gamma'$ applied to a boundary point $z$ returns a ray in $X\times Y$ which converges (in $\widehat{X\times Y})$ to $z$:
\Par
If $z=\left\langle \overline{x},0\right\rangle$, then $\gamma'(z)=(\alpha'(\overline{x},t),y_0)$ for all $t\geq 0$.  Since $\alpha'(\overline{x},t)\rightarrow\overline{x}$ in $\widehat{X}$ and $\displaystyle\mu(\gamma'(z))=\mu(\alpha'(\overline{x},t),y_0)=\frac{q(y_0)}{p(\alpha'(\overline{x},t))}=0$ for sufficiently large $t$, then $\gamma'(z)$ gets arbitrarily close to $\left\langle\overline{x},0\right\rangle$ in $\widehat{X\times Y}$.
\Par
A similar argument holds when $z=\left\langle \overline{y},\infty\right\rangle$.
\Par
Finally, if $z=\left\langle\overline{x},\overline{y},\mu\right\rangle$, where $0<\mu<\infty$, then for any $t\geq 0$, we have
\begin{equation*}\mu\left(\gamma'\left(\left\langle\overline{x},\overline{y},\mu\right\rangle,t\right)\right)=\displaystyle\frac{q\left(\beta'\left(\overline{y},\frac{\mu\cdot t}{\sqrt{\mu^2+1}}\right)\right)}{p\left(\alpha'\left(\overline{x},\frac{t}{\sqrt{\mu^2+1}}\right)\right)}\in\displaystyle\left(\frac{\frac{\mu\cdot t}{\sqrt{\mu^2+1}}-2}{\frac{t}{\sqrt{\mu^2+1}}+3},\frac{\frac{\mu\cdot t}{\sqrt{\mu^2+1}}+3}{\frac{t}{\sqrt{\mu^2+1}}-2}\right)\end{equation*}

\begin{center}$=\displaystyle\left(\frac{\mu\cdot t-2\sqrt{\mu^2+1}}{t+3\sqrt{\mu^2+1}},\frac{\mu\cdot t+3\sqrt{\mu^2+1}}{t-2\sqrt{\mu^2+1}}\right)$\end{center}
Therefore $\mu\left(\gamma'\left\langle\overline{x},\overline{y},\mu\right\rangle,t\right)\rightarrow\mu$ as $t\rightarrow\infty$, so that  $\gamma'\left(\left\langle\overline{x},\overline{y},\mu\right\rangle,t\right)\rightarrow\left\langle\overline{x},\overline{y},\mu\right\rangle$ in $\widehat{X\times Y}$ as $t\rightarrow\infty$.

This allows us to define $\gamma$, which begins at $\text{id}_{\widehat{X\times Y}}$ and then runs $\gamma'$ in reverse, and get a continuous map in doing so:
\Par
Let $\gamma:\widehat{X\times Y}\times[0,1]\rightarrow\widehat{X\times Y}$ be defined by
\begin{displaymath}
   \gamma(z,t) := \left\{
     \begin{array}{lr}
       z & \text{if }t=0\\
       \gamma'\left(z,\delta^{-1}(t)\right) & \text{if }t\in(0,1]
     \end{array}
   \right.
\end{displaymath}

\begin{proposition}\label{ER}  $\widehat{X\times Y}$ is an ER.\end{proposition}

\begin{proof}  Let $\mathcal{U}=\left\{U_{\alpha}\right\}_{\alpha\in A}$ be an open cover of $\widehat{X\times Y}$.  Choose $\epsilon>0$ such that for each $z\in\widehat{X\times Y}$ there is an $\alpha\in A$ so that $\gamma\left(\left\{z\right\}\times[0,\epsilon]\right)\subseteq U_{\alpha}$.  Note that such an $\epsilon$ exists since $\widehat{X\times Y}$ is compact.
\Par
Also note that $X\times Y$ is an ANR, being a product of ANR's.  We will show that $X\times Y$ is a $\mathcal{U}$-dominating space for $\widehat{X\times Y}$.
\Par
Consider the map $\psi:=\gamma_{\epsilon}:\widehat{X\times Y}\rightarrow X\times Y$, along with $\phi:X\times Y\hookrightarrow\widehat{X\times Y}$,  where $\phi$ is the inclusion map.
\Par
Then $\phi\circ\psi:\widehat{X\times Y}\rightarrow\widehat{X\times Y}$ is $\mathcal{U}$-homotopic to $\text{id}_{\widehat{X\times Y}}$ via the homotopy $\gamma |_{\widehat{X\times Y}\times [0,\epsilon]}$.
\Par
Thus $X\times Y$ is a $\mathcal{U}$-dominating space for $\widehat{X\times Y}$.
\Par
Hence $\widehat{X\times Y}$ is an ANR by Theorem \ref{hanner}.  Since $\widehat{X\times Y}$ is also contractible, then $\widehat{X\times Y}$ is an ER (Recall Fact \ref{ARER}).
\end{proof}

\begin{proposition}  \label{zset}$\partial X\ast\partial Y$ is a $\mathcal{Z}$-set in $\widehat{X\times Y}$.\end{proposition}

\begin{proof}  By construction, we have $\gamma_0\equiv\text{id}_{\widehat{X\times Y}}$ and $\gamma_t(\widehat{X\times Y})\cap\partial X\ast\partial Y=\emptyset$ whenever $t>0$.
\Par
Therefore $\partial X\ast\partial Y$ is a $\mathcal{Z}$-set in $\widehat{X\times Y}$.
\end{proof}

\begin{theorem}  \label{zstructurethmdirprod}Let $G$ and $H$ be groups which admit $\mathcal{Z}$-structures $(\widehat{X},\partial X)$ and $(\widehat{Y},\partial Y)$, respectively.  Then $(\widehat{X\times Y},\partial X\ast\partial Y)$ is a $\mathcal{Z}$-structure on $G\times H$.\end{theorem}

\begin{proof}  Propositions \ref{ER}, \ref{zset}, and \ref{null} show that conditions (1), (2), and (4) in Definition \ref{zstructure} are satisfied by $\widehat{X\times Y}$.  Moreover, $G\times H$ acts properly and cocompactly on $X\times Y$, since each of $G$ and $H$ acts accordingly on each of $X$ and $Y$, so condition (3) is also satisfied.
\Par
Therefore $(\widehat{X\times Y},\partial X\ast\partial Y)$ is a $\mathcal{Z}$-structure on $G\times H$.
\end{proof}

\begin{theorem}  \label{ezstructuredirprod}If $G$ and $H$ each admit $\mathcal{EZ}$-structures, then so does $G\times H$.\end{theorem}

\begin{proof}  By Theorem \ref{zstructurethmdirprod}, it suffices to show that the action of $G\times H$ on $X\times Y$ extends to an action on $\widehat{X\times Y}$.
\Par
By hypothesis, the actions of $G$ and $H$ on $X$ and $Y$ extend to actions on $\widehat{X}$ and $\widehat{Y}$, respectively, i.e. we have maps $\pi:G\times \widehat{X}\rightarrow \widehat{X}$ and $\overline{\rho}:H\times \widehat{Y}\rightarrow \widehat{Y}$ which satisfy the axioms of a group action.
\Par
Define the action of $G\times H$ on $\widehat{X\times Y}$ via the map
\\ $\overline{\tau}:(G\times H)\times\widehat{X\times Y}\rightarrow\widehat{X\times Y}$, where
\begin{align*}\overline{\tau}((g,h),(x,y)):=(\pi(g,x),\overline{\rho}(h,y))
\\
\overline{\tau}((g,h),\left\langle\overline{x},\overline{y},\mu\right\rangle):=(\pi(g,\overline{x}),\overline{\rho}(h,\overline{y}),\mu)
\\
\overline{\tau}((g,h),\left\langle\overline{x},0\right\rangle):=(\pi(g,\overline{x}),0)
\\
\overline{\tau}((g,h),\left\langle\overline{y},\infty\right\rangle):=(\overline{\rho}(h,\overline{y}),\infty)\end{align*}
\end{proof}

\section{Applications and Open Questions}

It is known that groups within certain classes admit $\mathcal{Z}$-structures, such as CAT(0), hyperbolic, and systolic groups.  However, it is more often than not the case that the direct product of two hyperbolic groups is not hyperbolic and that the direct product of two systolic groups is not systolic.  In addition, it is not clear how to handle the product (direct or free) of two groups when they come from distinct classes.  Theorems \ref{zstructurethmfreeprod} and \ref{zstructurethmdirprod} imply the following:

\begin{corollary}  Let $\mathcal{F}$ denote the family of groups consisting of all CAT(0), hyperbolic,  and systolic groups.  If $G,H\in \mathcal{F}$, then $G\ast H$ and $G\times H$ both admit $\mathcal{EZ}$-structures.\end{corollary}

We end the paper with some open questions related to this work:
\Par
(1)  Does a modification of the construction in the proof of Theorem \ref{zstructurethmfreeprod} give an analogous result pertaining to free products with amalgamation over finite subgroups?
\Par
(2)  Does a variation of Theorem \ref{zstructurethmfreeprod} hold for HNN extensions over finite subgroups?
\Par
(3)  If $G$, $H$, and $K$ all admit $\mathcal{Z}$-structures, does $G\ast_{K}H$ admit a $\mathcal{Z}$-structure?  What about $G\ast_K$, again under the hypothesis that $G$ and $K$ admit $\mathcal{Z}$-structures?

\pagestyle{myheadings}
\bibliographystyle{ieeetr}
\bibliography{Bibliography} \thispagestyle{myheadings}
\addcontentsline{toc}{chapter}{\ \quad Bibliography}
\end{document}